\documentclass[11pt,a4paper,reqno]{amsart}
\usepackage[a4paper,left=3cm,right=3cm,bottom=4cm,top = 3.5cm]{geometry}

\usepackage[pdftex,bookmarks=true,pdfborder={0 0 0}]{hyperref}

\usepackage[foot]{amsaddr}

\usepackage[utf8]{inputenc}
\usepackage{graphicx}
\usepackage{amsmath}
\usepackage{amssymb}
\usepackage{mathtools}
\usepackage{enumitem}
\usepackage{color}
\usepackage{microtype}
\usepackage{cite}

\usepackage[nameinlink]{cleveref}

\usepackage{amsthm}
\newtheorem{theorem}{Theorem}[section]
\newtheorem{lemma}[theorem]{Lemma}

\newtheorem{proposition}[theorem]{Proposition}

\theoremstyle{definition}

\newtheorem{problem}[theorem]{Problem}
\newtheorem{remark}[theorem]{Remark}
\newtheorem*{remark*}{Remark}

\newcommand{\beq}{\begin{equation}}
\newcommand{\eeq}{\end{equation}}

\newcommand{\lem}[1]{{Lemma\,#1}}
\newcommand{\pr}[1]{{Proposition\,#1}}
\newcommand{\theo}[1]{{Theorem\,#1}}

\newcommand{\abs}[1]{\left\vert#1\right\vert}

\newcommand{\R}{\mathbb{R}}

\newcommand{\N}{\mathbb{N}}

\newcommand{\Mr}{\mathcal{M}_r}
\newcommand{\bfr}{\mathbf{r}}
\newcommand{\bfk}{\mathbf{k}}

\renewcommand{\Mc}{\mathcal{M}}

\providecommand{\abs}[1]{\lvert#1\rvert}

\providecommand{\norm}[1]{\lVert#1\rVert}

\DeclareMathOperator{\rank}{rank}

\DeclareMathOperator{\dist}{dist}

\DeclareMathOperator{\spn}{span}

\newtheorem*{thm*}{Theorem}
\newtheorem*{lm*}{Lemma}

\theoremstyle{definition}

\newtheorem*{ques*}{Question}
\newtheorem*{defi*}{Definition}
\newtheorem*{exam*}{Example}
\newtheorem*{exams*}{Examples}

\newcommand{\T}{\mathsf T}

\DeclareMathOperator{\id}{id}

\DeclareMathOperator{\spa}{span}

\numberwithin{equation}{section}

\title[Dynamical low-rank tensor approximations to parabolic problems]{Dynamical low-rank tensor approximations to high-dimensional parabolic problems: existence and convergence of spatial discretizations}

\author{Markus~Bachmayr$^1$}
\email{bachmayr@igpm.rwth-aachen.de}

\author{Henrik~Eisenmann$^1$}
\email{eisenmann@igpm.rwth-aachen.de}

\author{Andr\'e~Uschmajew$^2$}
\email{andre.uschmajew@uni-a.de}

\address{$^1$ Institut f\"ur Geometrie und Praktische Mathematik, RWTH Aachen University, Templergraben 55, 52062 Aachen, Germany} \address{$^2$ Institute of Mathematics \& Centre for Advanced Analytics and Predictive Sciences, University of Augsburg, 86159 Augsburg, Germany}


\begin{document}

\begin{abstract}
We consider dynamical low-rank approximations to parabolic problems on higher-order tensor manifolds in Hilbert spaces. In addition to existence of solutions and their stability with respect to perturbations to the problem data, we show convergence of spatial discretizations. Our framework accommodates various standard low-rank tensor formats for multivariate functions, including tensor train and hierarchical tensors.
\end{abstract}

\maketitle

\section{Introduction}

Dynamical low-rank approximation (DLRA) is a nonlinear model for time evolution of high-dimensional functions on low-dimensional submanifolds. In its simplest form for matrices, as presented in the seminal work~\cite{Koch07}, one seeks approximate solutions of matrix valued ODEs
\[
\dot X(t) = F(X(t),t), \quad X(t) \in \R^{M \times N}
\]
constrained to a low-rank model
\[
X(t) = U(t) V(t)^\T,
\]
where $U(t) \in \R^{M \times r}$ and $V(t) \in \R^{N \times r}$. Hence $X(t)$ has rank at most $r$ for every time~$t$, which if $r \ll M,N$ allows for an efficient numerical treatment. From a geometric perspective, by enforcing $X(t)$ to have rank exactly $r$, the problem can be formulated as an ODE on the manifold $\Mr$ of matrices of fixed rank $r$ by projecting the vector field $F$ onto the tangent space $T_{X(t)} \Mr$ at every time~$t$,
\begin{equation}\label{eq: projected formulation}
\dot X(t) = P_{X(t)} F(X(t),t).
\end{equation}
In this way, assuming a starting value $X_0$ on $\Mr$, the time evolution is automatically constrained to the manifold. This constrained problem admits the time-dependent variational formulation
\begin{equation}\label{eq: variational formulation}
\langle \dot X(t), Y \rangle = \langle F(X(t),t) , Y \rangle \quad \text{for all $Y \in T_{X(t)} \Mr$,}
\end{equation}
which in physics is also called the Dirac--Frenkel principle \cite{Dirac30,Frenkel1934}. The state-of-the-art tool for the numerical solution of these equations is based on a splitting of the tangent space projector $P_{X(t)}$, see~\cite{Lubich14}.

The DLRA approach can also be applied to time-dependent problems where the solution $X(t)$ is a tensor of size $N_1 \times \dots \times N_d$. In this case, various low-rank models are possible, notably those based on tree tensor networks~\cite{Koch10,Lubich13,Arnold14}.
The most basic example is the low-rank Tucker format 
\[
  X(t;i_1,\dots,i_d) = \sum_{j_1=1}^{r_1}\cdots \sum_{j_d=1}^{r_d}C(t;j_1,\dots, j_d) U_1(t;i_1,j_1)\cdots U_d(t;j_d,i_d) 
\]
with multilinear rank $\bfr = (r_1,\ldots, r_d)$.
Another example is the widely used tensor train model~\cite{Oseledets11}, also known as matrix product states in physics, in which the tensor $X(t)$ element-wise reads as
\begin{equation}
  X(t;i_1,\dots,i_d) = G_1(t;i_1) G_2(t;i_2) \cdots G_d(t;i_d)\,.
\label{eq: TT-format}  
\end{equation}
Here $G_\mu(t,i_\mu)$  for $\mu=1,\ldots,d$ is a matrix of fixed size $k_{\mu-1} \times k_{\mu}$, with the convention that $k_0 = k_d = 1$. Hence, the parameters in this model are the so-called core tensors $G_\mu(t) \in \R^{k_{\mu-1} \times N_\mu \times k_\mu}$. The componentwise minimal possible tuple $\bfk = (k_1,\dots,k_{d-1})$ is called the TT rank of $X$, and the set of tensors with a fixed TT rank $\bfk = (k_1,\dots,k_{d-1})$ can be shown to form an embedded manifold $\mathcal M_\bfk$ in $\R^{N_1\times \dots \times N_d}$, see~\cite{Holtz12,UschmajewVandereycken:13}. Note that for $d=2$, one recovers matrices of rank $k = k_1$. The DLRA formulation then essentially takes the same form~\eqref{eq: projected formulation} or~\eqref{eq: variational formulation} with the according tangent space projection, and the resulting problem again can be solved numerically using a projector splitting approach~\cite{Lubich15,Ceruti21}. We refer to~\cite{UschmajewVandereycken:20} for a survey on DLRA for matrix and tensor problems.

So far we have outlined DLRA for problems in finite-dimensional matrix and tensor spaces. In applications of DLRA related to PDE problems, such as in \cite{Meyer90,Bardos09,Bardos10,Conte10,Arnold14,Lubich15a,Einkemmer18,Einkemmer:20,Kazashi20,Guo:22}, a finite-dimensional problem appears as the second step after a spatial discretization of multivariate functions. Then the natural question arises whether the solutions of discrete problems converge to a meaningful solution of a continuous DLRA formulation when the spatial discretization is refined. To the best of our knowledge, this convergence issue has not been addressed in the existing literature. Since DLRA is based on a nonlinear variational formulation, the convergence of spatial discretizations is a nontrivial matter and cannot be treated by classical results from the numerical analysis of PDEs.

\subsection{Dynamical low-rank approximation for parabolic problems}

In this work, we establish such results for certain types of parabolic PDEs. As a model, we consider the diffusion equation 
\begin{equation}\label{eq:modelproblem1}
\begin{aligned}
    \frac{\partial}{\partial t} u(x,t) -\nabla_x\cdot (B(t) \nabla_x u(x,t)) &= f(x,t)&\quad&\text{for } (x,t)\in \Omega\times (0,T),
    \\
    u(x,t)&=0 &\quad&\text{for } (x,t)\in \partial\Omega\times (0,T),
    \\
    u(x,0)&=u_0(x) &\quad&\text{for } x\in \Omega
\end{aligned}
\end{equation}
on the product domain $\Omega=(0,1)^d$. Here $B(t)$ is a $d\times d$-matrix, which allows for anisotropic diffusion in \eqref{eq:modelproblem1}, and we assume it to be uniformly bounded and positive definite, as well as Lipschitz continuous with respect to $t$. The problem~\eqref{eq:modelproblem1} is typically formulated in weak form on function spaces as follows: given $u_0 \in L_2(\Omega)$ and $f\in L_2(0,T; H_{}^{-1}(\Omega))$, find
\[
u\in W(0,T;H_0^1(\Omega), H_{}^{-1}(\Omega)) = \{u\in L_2(0,T; H_0^1(\Omega))\colon u'\in L_2(0,T; H_{}^{-1}(\Omega)) \}
\]
such that for almost all $t\in(0,T)$,
\begin{equation}\label{eq:modelproblem2}
\begin{aligned}
    \langle u'(t), v\rangle +a( u(t), v;t)&=\langle f(t), v\rangle\quad \text{for all $ v\in H_0^1(\Omega) $,}    \\
    u(0)&=u_0.
\end{aligned}
\end{equation}
Here, by $\langle\cdot,\cdot \rangle$ we denote the dual pairing on $ H_{}^{-1}(\Omega)\times  H_{0}^{1}(\Omega)$, and $a\colon  H_{0}^{1}(\Omega)\times  H_{0}^{1}(\Omega)\times [0,T] \to \R$ is the bounded, symmetric and coercive bilinear form
\begin{equation*}
a(u,v;t) = \int_\Omega (B(t)\nabla u(x))\cdot \nabla v(x) \,dx.
\end{equation*}
Classical theory provides a unique solution to~\eqref{eq:modelproblem2}; see for example~\cite[\theo 23.A]{Zeidler90a} and~\cite[\theo 30.A]{Zeidler90b}.

A DLRA formulation for~\eqref{eq:modelproblem2} is obtained by selecting a low-rank model class
\[
\mathcal M \subset L_2(\Omega)
\]
based on the tensor product structure of $L_2(\Omega) = L_2(0,1) \otimes \dots \otimes L_2(0,1)$, such as the tensor train format in a Hilbert space setting to be made precise in Section~\ref{sec:tensor train model}, and restricting the test function in~\eqref{eq:modelproblem2} to the tangent space of $\mathcal M$ at $u(t)$. In other words, the problem reads: given $u_0 \in \mathcal M$ and $f\in L_2(0,T; H_{}^{-1}(\Omega))$, find
\[
u\in W(0,T;H_0^1(\Omega), H_{}^{-1}(\Omega))\coloneqq\{u\in L_2(0,T; H_0^1(\Omega))\colon u'\in L_2(0,T; H_{}^{-1}(\Omega)) \}
\]
such that $u(t) \in \mathcal M$ for all $t\in(0,T)$ and
\begin{equation}\label{eq:modelproblem2DLRA}
\begin{aligned}
    \langle u'(t), v\rangle +a( u(t), v;t)&=\langle f(t), v\rangle\quad \text{for all $ v\in T_{u(t)} \mathcal M \cap H_0^1(\Omega) $,}    \\
    u(0)&=u_0.
\end{aligned}
\end{equation}

For the matrix case $d=2$, existence and uniqueness of DLRA solutions on a maximal time interval $(0,T^*)$ has been shown in~\cite{Bachmayr21} under the additional regularity assumptions $u_0 \in H^1_0(\Omega)$ and $f \in L_2(0,T; L_2(\Omega))$. The proof given there is carried out in a more general framework of parabolic problems on manifolds in Gelfand triplets and is based on a variational time stepping scheme in Hilbert space. As we will verify in \Cref{sec:tensor train model}, this framework also applies to the tensor train model of $d$-variate functions, and can analogously be used to other tree based low-rank tensor models. To this end, we establish a general concept of tensor manifolds in Hilbert space which allows to deduce the required manifold properties, such as curvature estimates, from their finite-dimensional counterparts. As a result, we obtain a meaningful and rigorous notion of a continuous DLRA solution for parabolic problems for higher-order tensors based on the results from~\cite{Bachmayr21}.

However, the existence proof in~\cite{Bachmayr21} does not provide the convergence of solutions of spatial semidiscretizations to the continuous solution, and we address this open problem in the present work as well. Spatial discretizations of low-rank problems can be obtained in a very natural way via tensor products of discretization spaces. By considering finite-dimensional subspaces $V^\mu_h \subset H_0^1(0,1)$ of dimension $N_\mu$ one obtains a DLRA problem analogous to~\eqref{eq:modelproblem2DLRA} by restricting to the tensor product space
\[
\mathcal V_h = V^1_h \otimes \dots \otimes V^d_h.
\]
We hence seek a solution $u_h(t) \in \mathcal M \cap \mathcal V_h$ satisfying
\begin{equation}\label{eq:modelproblem2DLRAdiscrete}
\begin{aligned}
    \langle u_h'(t), v_h \rangle +a( u_h(t), v_h;t)&=\langle f(t), v_h \rangle\quad \text{for all $ v_h\in T_{u_h(t)} \mathcal M \cap \mathcal V_h$, }   \\
    u(0)&=u_0.
\end{aligned}
\end{equation}
By writing $u_h(t)$ in a tensor product basis of $\mathcal V_h$,
\[
u_h(t) = \sum_{i_1 = 1}^{N_1} \cdots \sum_{i_d = 1}^{N_d} X(t;i_1,\dots,i_d) \varphi^1_{i_1} \otimes \dots \otimes \varphi^d_{i_d}\,,
\]
we obtain a DLRA problem for the coefficient tensor $X(t) \in \R^{N_1 \times \dots \times N_d}$. The question is then whether for $h \to 0$, the functions $u_h$ converge to the (unique) solution $u$ of~\eqref{eq:modelproblem2DLRAdiscrete}. 

\subsection{Contributions and outline}

In this work, we extend the existence and uniqueness result for the DLRA evolution obtained in~\cite{Bachmayr21} from the bivariate to the multivariate case.
The abstract framework from~\cite{Bachmayr21} is recalled in Section~\ref{sec:abstproblem}. This section also contains the new stability estimate Theorem~\ref{thm:stability}, which complements the uniqueness result of~\cite{Bachmayr21}.
The main result of Section~\ref{sec:spatialdiscrete} is the 
convergence result for spatial semidiscretizations (Theorem~\ref{thm:convergence_space-discr}) in the abstract framework of Gelfand triplets as developed in~\cite{Bachmayr21}. In Section~\ref{sec:tensor train model}, we present a general notion of low-rank manifolds in Hilbert spaces and obtain new curvature estimates. We then in Section~\ref{sec:applmodelproblem} apply these results to dynamical low-rank approximations of the model problem~\eqref{eq:modelproblem2DLRA} using the tensor train format.
Based on the general setting of Section~\ref{sec:tensor train model}, analogous results can be obtained for other low-rank manifolds; in particular, our results apply directly also to the Tucker format.
Note that in this work, we focus on a theoretical framework for DLRA of parabolic problems and do not consider specific numerical realizations. Numerical results for DLRA in the parabolic case can be found, e.g., in \cite{Nonnenmacher08,Kazashi20}.

\section{Abstract formulation}\label{sec:abstproblem}

For developing the essential aspects of the theory we investigate the DLRA problem in an abstract context as in~\cite{Bachmayr21}. We consider a Gelfand triplet
\[
\mathcal V\hookrightarrow \mathcal H\cong \mathcal H^* \hookrightarrow\mathcal V^*
\]
of Hilbert spaces, where $ \mathcal V$ is compactly embedded in $ \mathcal H$. This implies that the embedding is also continuous, that is,
\begin{equation}\label{eq:norminequality}
    \|u\|_\mathcal H\lesssim \|u\|_\mathcal V.
\end{equation}
 By $\langle\cdot,\cdot\rangle : \mathcal V^* \times \mathcal V \to \R$ we denote the dual pairing of $\mathcal V^*$ and $\mathcal V$.
Note that for $u\in\mathcal H\subset \mathcal V^*$ and $v\in \mathcal V\subset \mathcal H$, the dual pairing and the inner product on $\mathcal H$ coincide, that is, $\langle u,v\rangle_{\mathcal H}=\langle u,v\rangle$. We will frequently identify $u\in \mathcal V$ as an element of $\mathcal H$ and in turn also as an element in~$\mathcal V^*$.

For every $t \in [0,T]$, let $a(\cdot,\cdot;t): \mathcal V \times \mathcal V \to \R$ be a bilinear form which is assumed to be symmetric,
\[
  a(u,v;t) = a(v,u;t) \quad \text{for all } u, v \in \mathcal V \text{ and } t \in [0,T],
\]
uniformly bounded,
\[
  |a(u,v;t)| \le \beta\|u\|_\mathcal V \|v\|_\mathcal V \quad \text{for all } u, v \in \mathcal V \text{ and } t \in [0,T]
\]
for some $\beta > 0$, and uniformly coercive,
\[
  a(u,u;t) \ge \mu \|u\|_\mathcal V^2 \quad \text{for all } u \in \mathcal V \text{ and } t \in [0,T]
\]
for some $\mu > 0$. Under these assumptions, $a(\cdot,\cdot;t)$ is an inner product on $\mathcal V$ defining an equivalent norm. Furthermore, it defines a bounded operator
\begin{equation}\label{eq: associated unbounded operator}
 A(t) : \mathcal V \to \mathcal  V^*
\end{equation}
such that
\[
a(u,v;t) = \langle A(t)u,v \rangle \quad \text{for all $u,v \in \mathcal V$.}
\]

We also assume that $a(u,v;t)$ is Lipschitz continuous with respect to $t$. In other words, there exists an $L \ge 0$ such that
\begin{equation}
  \label{eq:a-lipschitz}
  |a(u,v; t) - a(u,v; s)| \le L\beta\|u\|_\mathcal V\|v\|_\mathcal V |t - s|
\end{equation}
for all $u, v \in \mathcal V$ and $s, t \in [0, T]$, which in the model problem corresponds to the Lipschitz continuity of the function $t \mapsto B(t)$. 

We deal with evolution equations on subsets $\mathcal M\subset \mathcal H$ that are submanifolds in the following sense: for every point $u\in\mathcal M$ 
there exists a closed subspace $T_u\mathcal M \subset\mathcal H$ 
such that $T_u\mathcal M$ contains all tangent vectors to $\mathcal M$ at $u$. Here a tangent vector is any $v \in \mathcal H$ for which there exists a (strongly) differentiable curve $\varphi\colon (-\varepsilon, \varepsilon) \to \mathcal H$ (for some $\varepsilon > 0$) such that $\varphi(t) \in \mathcal M$ for all $t$ and
\[
 \varphi(0) = u, \quad \varphi'(0) = v.
\]
By $P_u:\mathcal H\to T_u\mathcal M$ we denote the $\mathcal H$-orthogonal projection onto $T_u\mathcal M$. We will also assume $\mathcal M\cap \mathcal V$ to be nonempty as well as  $T_u\mathcal M\cap \mathcal V$ to be nonempty for $u\in\mathcal M\cap \mathcal V$. By $\overline{\mathcal M}^w$ we denote the weak closure of $\mathcal M$ in $\mathcal H$.

The abstract problem takes the following form. 
\begin{problem}\label{problem 1}
In the above setting, given $f\in L_2(0,T;\mathcal V^*)$ and
$u_0 \in \mathcal M \cap \mathcal H$, find
\[
 u \in W(0,T;\mathcal V,\mathcal V^*) \coloneqq \{ u \in L_2(0,T; \mathcal V) \colon  u' \in L_2(0,T; \mathcal V^*) \}
\]
such that for almost all $t \in [0,T]$,
\begin{equation}\label{eq:prob1}
  \begin{aligned}
  u(t)  &\in \mathcal M,\\
    \langle u'(t), v\rangle + a(u(t),v;t) &= \langle f(t), v \rangle \quad \text{for all $v \in T_{u(t)} \mathcal M \cap \mathcal V$}, \\
    u(0) &= u_0.
  \end{aligned}
\end{equation}
\end{problem}

\subsection{Basic assumptions and existence of solutions}

The main challenge of the weak formulation in Problem \ref{problem 1} is that according to the Dirac-Frenkel principle, the test functions are from the tangent space only. The existence result in~\cite{Bachmayr21} requires several assumptions, including additional regularity of the data as in assumption~\ref{property:regularity} below. The other assumption~\ref{property:cone}--\ref{property:splitting} are abstractions of corresponding properties of the model problem and will be discussed for the tensor train format in \Cref{sec:tensor train model}, and hence the main results of this paper apply to this setting. The assumptions are the following.

\begin{enumerate}[label=\bf{A\arabic*},start=0,itemsep=5pt]
\item\label{property:regularity} (Regularity of data) We have
$f\in L_2(0,T;\mathcal H)$ and $u_0 \in \mathcal M \cap \mathcal V$.
    \item\label{property:cone} (Cone property)
 $\mathcal M$ is a cone, that is, $u \in \mathcal M$ implies $s u \in \mathcal M$ for all $s > 0$. 
\item\label{property:curvature} (Curvature bound) 
For every subset $\mathcal M'$ of $\mathcal M$ that is weakly compact in $\mathcal H$, there exists  a constant $\kappa = \kappa(\mathcal M')$ such that
  \[
    \|P_u - P_v \|_{\mathcal H \to \mathcal H} \le \kappa \|u - v\|_\mathcal H 
  \]
and
\[
\|(I-P_u)(u-v)\|_\mathcal H\leq\kappa\|u-v\|_\mathcal H^2 
\]
for all $u,v \in \mathcal M'$.

\item\label{property:compat} (Compatibility of tangent spaces)
\begin{enumerate}[ref={\bf{A3}}(\alph*)]
    \item\label{property:compat1}
 For $u \in \mathcal M \cap \mathcal V$ and $v \in T_u \mathcal M \cap \mathcal V$ an admissible curve with $\varphi(0) = u$, $\varphi'(0) = v$ can be chosen such that
\[                     
\varphi(t) \in \mathcal M \cap \mathcal V
\]
for all $\abs{t}$ small enough.
\item\label{property:compat2}
  If $u \in \mathcal M \cap \mathcal V$ and $v \in \mathcal V$ then $P_u v \in T_u \mathcal M \cap \mathcal V$. 
\end{enumerate}
\item\label{property:splitting} (Operator splitting) The associated operator $A(t)$ in~\eqref{eq: associated unbounded operator} admits a splitting
\[
 A(t) = A_1(t) + A_2(t)
\]
into two uniformly bounded operators $\mathcal V\to \mathcal V^*$ 
such that for all $t \in [0,T]$, all $u \in \mathcal M \cap \mathcal V$ and all $v \in \mathcal V$, the following holds:
\begin{enumerate}[ref={\bf{A4}}(\alph*)]
\item\label{property:partA1}``$A_1(t)$ maps to the tangent space'':
\[
\langle A_1(t)u, v \rangle = \langle A_1(t) u, P_u v \rangle.
\]
 \item\label{property:partA2}``$A_2(t)$ is locally bounded from $\mathcal M \cap \mathcal V$ to $\mathcal H$'': For every subset $\mathcal M'$ of $\mathcal M$ that is weakly compact in $
\mathcal H$, there exists $\gamma = \gamma(\mathcal M') > 0$ such that
\[ 
A_2(t)u \in \mathcal H \quad \text{and} \quad \| A_2(t) u \|_\mathcal H \le \gamma \| u \|_\mathcal V^\eta \quad \text{for all $u \in \mathcal M'$}
\]
with an $\eta>0$ independent of $\mathcal M'$.
\end{enumerate}
\end{enumerate}

\medskip

The following existence and uniqueness result has been obtained in~\cite[Theorem~4.3]{Bachmayr21}. Here $W(0,T;\mathcal V,\mathcal H)$ denotes the subspace of $W(0,T;\mathcal V,\mathcal H)$ with $u' \in L_2(0,T; \mathcal H)$.

\begin{theorem}\label{thm:existandunique}
Let the Assumptions~\ref{property:regularity}--\ref{property:splitting} hold and let $u_0$ have positive $\mathcal H$-distance from $\overline{\mathcal M}^w \setminus \mathcal M$. There exist $T^* \in (0,T]$ and $u \in W(0,T^*;\mathcal V,\mathcal H) \cap L_\infty(0,T^*;\mathcal V)$ such that $u$ solves Problem~\ref{problem 1} on the time interval $[0,T^*]$, and its continuous representative $u\in C(0,T^*;\mathcal H)$ satisfies $u(t) \in \mathcal M$ for all $t \in [0,T^*)$. Here $T^*$ is maximal for the evolution on $\mathcal M$ in the sense that if $T^* < T$, then 
\[
  \liminf_{t \to T^*} \; \inf_{v\in\overline{\mathcal M}^w \setminus \mathcal M } \|u(t)-v\|_\mathcal H = 0.
\]
In either case, $u$ is the unique solution of Problem~\ref{problem 1} in $W(0,T^*;\mathcal V,\mathcal H) \cap L_\eta(0,T^*;\mathcal V)$.

In particular, let $\sigma = \dist_{\mathcal H}(u_0,\overline{\mathcal M}^w \setminus \mathcal M )$, then there exists a constant $c>0$ such that $T^* \ge \min(\sigma^2 / c,T)$.
 
The solution satisfies the following estimates:
\begin{align}
    \|u\|^2_{L_2(0,T^*;\mathcal V)}&\leq \|u_0\|_\mathcal H^2+C_1\|f\|^2_{L_2(0,T^*;\mathcal H)},\label{eq:estimate_u}
    \\
    \|u'\|^2_{L_2(0,T^*;\mathcal H)}&\leq C_2\left(\|u_0\|_\mathcal V^2+\|f\|^2_{L_2(0,T^*;\mathcal H)}\right),\label{eq:estimate_u'}
    \\
    \|u\|^2_{L_\infty(0,T^*;\mathcal V)}&\leq C_3\left(\|u_0\|_\mathcal V^2+\|f\|^2_{L_2(0,T^*;\mathcal H)}\right),\label{eq:estimate_u_infty}
\end{align}
where $C_1$, $C_2$, and $C_3$ are the constants from~\cite[Lemma~4.4]{Bachmayr21}. 
\end{theorem}

Note that the energy estimates~\eqref{eq:estimate_u}--\eqref{eq:estimate_u_infty} are not explicitly stated in~\cite[Theorem~4.3]{Bachmayr21}, but immediately follow from its proof in combination with~\cite[Lemma~4.4]{Bachmayr21}.

\begin{remark}
  We can take $c$ as the right-hand side of~\eqref{eq:estimate_u'}.
\end{remark}

\begin{remark}
  Our Assumption~\ref{property:curvature} is stronger than the one used in~\cite{Bachmayr21}, which requires the curvature estimates to be valid for $\norm{ u - v }_\mathcal{H} \leq \varepsilon$ for some $\varepsilon>0$ (the constant $\kappa$ then may depend on $\varepsilon$).
  The stronger assumption made here is used in the proof of Theorem~\ref{thm:stability}. As our new curvature estimates in Section~\ref{sec:curvature} show, this stronger assumption is in fact satisfied for the low-rank manifolds under consideration.
\end{remark}

\subsection{A stability estimate}

As a first extension to the above existence and uniqueness theorem, we now provide a stability estimate that in particular ensures continuity of the solution with respect to the data. This result was obtained in \cite{Eisenmann}. The proof follows a similar idea as the uniqueness result in~\cite[Theorem~4.1]{Bachmayr21}.

\begin{theorem}\label{thm:stability}
Let $u,v \in W(0,T^*;\mathcal V,\mathcal H)$ be two solutions of \Cref{problem 1} on a time interval $[0,T^*]$ corresponding to right-hand sides $f,g \in L_2(0,T;\mathcal H)$, and initial values $u_0, v_0 \in \mathcal M$, respectively. Assume that the continuous representatives $u,v\in C(0,T^*;\mathcal H)$ have values in a weakly compact subset $\mathcal M'\subset \mathcal M$ (in particular their $\mathcal H$-distance to $\overline{\mathcal M}^w \setminus \mathcal M$ remains bounded from below). Moreover, assume that $u,v \in L_\eta(0,T^*;\mathcal V)$ where $\eta$ is from Assumption~{\upshape \ref{property:partA2}}. 
Then for any $\varepsilon >0$,
\[
\|u(t)-v(t)\|_\mathcal H^2\leq \left(\|u_0-v_0\|_\mathcal H^2 +\frac{1}{\varepsilon}\int_0^t\|f(s)-g(s)\|_\mathcal H^2\, ds\right)\exp(\Lambda(t)+\varepsilon t),
\]
where 
\[
\Lambda(t)\coloneqq 2\kappa\int_0^t \|u'(s)\|_\mathcal H+\|v'(s)\|_\mathcal H
+
\gamma\left(\|u(s)\|_\mathcal V^\eta+\|v(s)\|_\mathcal V^\eta\right)
+
\|f(s)\|_\mathcal H+\|g(s)\|_\mathcal H\, ds
 < \infty
\]
with $\kappa=\kappa(\mathcal M')$ from Assumption~\ref{property:curvature}.
\end{theorem}

\begin{proof}
We use integration by parts in the sense of~\cite[\pr 23.23(iv)]{Zeidler90a}. This results in
\[
\frac{1}{2}\frac{d}{dt}\|u(t)-v(t)\|_\mathcal H^2
\leq
\langle u'(t)-v'(t)+A(t)(u(t)-v(t))-f(t)+g(t)+f(t)-g(t),u(t)-v(t)\rangle
\]
for almost all $t\in[0,T^*]$ 
by coercivity of $A(t)$ and adding and subtracting $\langle f(t)-g(t),u(t)-v(t)\rangle$. We add and subtract~\eqref{eq:prob1} for the solutions $u$ and $v$ with $w=P_{v(t)}(u(t)-v(t))$ and $w=P_{u(t)}(u(t)-v(t))$, respectively. This results in
\begin{multline*}
    \frac{1}{2}\frac{d}{dt}\|u(t)-v(t)\|_\mathcal H^2\\
\leq
\langle f(t)-g(t),u(t)-v(t)\rangle
+
\langle u'(t)+A(t)u(t)-f(t),(\id-P_{u(t)})(u(t)-v(t))\rangle\\
-\langle v'(t)+A(t)v(t)-g(t),(\id-P_{v(t)})(u(t)-v(t))\rangle.
\end{multline*}
We use Young's inequality to estimate
\[
\langle f(t)-g(t),u(t)-v(t)\rangle\leq \frac{1}{2\varepsilon}\|f(t)-g(t)\|_\mathcal H^2 +\frac{\varepsilon}{2}\|u(t)-v(t)\|_\mathcal H^2
\]
and Assumption~\ref{property:splitting} to get
\begin{multline*}
    \frac{1}{2}\frac{d}{dt}\|u(t)-v(t)\|_\mathcal H^2\\
\leq
\left(
\|u'(t)\|_\mathcal H+\gamma\|u(t)\|_\mathcal V^\eta+\|f(t)\|_\mathcal H
\right)
\|(\id-P_{u(t)})(u(t)-v(t))\|_\mathcal H\\
\quad
+
\left(
\|v'(t)\|_\mathcal H+\gamma\|v(t)\|_\mathcal V^\eta+\|g(t)\|_\mathcal H
\right)
\|(\id-P_{v(t)})(u(t)-v(t))\|_\mathcal H\\
+\frac{1}{2\varepsilon}\|f(t)-g(t)\|_\mathcal H^2 +\frac{\varepsilon}{2}\|u(t)-v(t)\|_\mathcal H^2.
\end{multline*}
Finally, Assumption~\ref{property:curvature} implies
\begin{multline*}
\frac{d}{dt}\|u(t)-v(t)\|_\mathcal H^2
\leq
\biggl(2\kappa\Bigl(
\|u'(t)\|_\mathcal H+\|v'(t)\|_\mathcal H+\gamma\bigl(\|u(t)\|_\mathcal V^\eta+\|u(t)\|_\mathcal V^\eta\bigr)
\\  +
\|f(t)\|_\mathcal H+\|g(t)\|_\mathcal H
\Bigr) +\varepsilon\biggr)
\|u(t)-v(t)\|^2_\mathcal H
+\frac{1}{\varepsilon}\|f(t)-g(t)\|_\mathcal H^2
\end{multline*}
and the result follows from Gr\"onwall's lemma; see for example~\cite[\lem 2.7]{Teschl12}. 
Here we take into account that $L_2(0,T^*;\mathcal H)\subset L_1(0,T^*;\mathcal H)$.
\end{proof}

\section{Convergence of spatial discretizations}\label{sec:spatialdiscrete}

From the perspective of numerical analysis, an important question is whether the unique solution of \Cref{problem 1} can be obtained as the limit of solutions of spatially discretized problems. We now provide such a result under  assumptions on the compatibility of the discrete spaces $\mathcal V_h\subset \mathcal V$ with $\mathcal M$. 
The spatially discretized problems are of the following form.

\begin{problem}
\label{problem discrete}
Given $f\in L_2(0,T;\mathcal H)$ and
$u_{0,h} \in \mathcal M \cap \mathcal V_h$, find
$ u_h \in W(0,T;\mathcal V,\mathcal H) $
such that for almost all $t \in [0,T]$,
\begin{equation}\label{eq:prob_discr}
  \begin{aligned}
  u_h(t)  &\in \mathcal M\cap \mathcal V_h,\\
    \langle u_h'(t), v_h\rangle + a(u_h(t),v_h;t) &= \langle f(t), v_h \rangle \quad \text{for all $v_h \in T_{u_h(t)} \mathcal M \cap \mathcal V_h$}, \\
    u_h(0) &= u_{h,0}.
  \end{aligned}
\end{equation}
\end{problem}

We require that the discrete subspaces $\mathcal V_h\subset \mathcal V$ have the following properties.

\begin{enumerate}[label=\bf{B\arabic*},itemsep=5pt]
\item \label{property:discr_approximation} (Approximation property) 
\begin{enumerate}[ref={\bf{B1}}(\alph*)]
  \item \label{property:discr_approximation1} 
  For every $v \in \mathcal V$, the $\mathcal V$-orthogonal projections $v_h \in \mathcal V_h$ satisfy $\| v - v_h \|_{\mathcal V} \to 0$ as $h \searrow 0$.
\item \label{property:discr_approximation2} For every $u\in\mathcal M\cap\mathcal V$ there is a sequence $(u_h)$ with $u_h\in\mathcal M\cap\mathcal V_h$ such that $u_h$ converges to $u$ in $\mathcal V$ as $h\searrow 0$ and $\|u_h\|_\mathcal V \leq \|u\|_\mathcal V$. 
\end{enumerate}

\item \label{property:discr_compat} (Compatibility of tangent spaces)
\begin{enumerate}[ref={\bf{B2}}(\alph*)]
    \item\label{property:discr_compat1}
 For $u_h \in \mathcal M \cap \mathcal V_h$ and $v_h \in T_u \mathcal M \cap \mathcal V_h$ a continuously differentiable curve with $\varphi(0) = u_h$, $\varphi'(0) = v_h$ can be chosen such that
\[                     
\varphi(t) \in \mathcal M \cap \mathcal V_h
\]
for all $\abs{t}$ small enough.
\item\label{property:discr_compat2}
  If $u_h \in \mathcal M \cap \mathcal V_h$ and $v_h \in \mathcal V_h$ then $P_{u_h} v_h \in T_{u_h} \mathcal M \cap \mathcal V_h$. 
\end{enumerate}
\end{enumerate}

The model problem \eqref{eq:modelproblem2} allows for such a space discretization, as verified in \Cref{sec:assumptionsBmodel} for tensor train manifolds. The following result was obtained in \cite{Eisenmann}.

\begin{theorem}\label{thm:convergence_space-discr}
Let the Assumptions~\ref{property:regularity}--\ref{property:splitting} and~\ref{property:discr_approximation}--\ref{property:discr_compat} hold. Let $u_{0,h}\in \mathcal M \cap \mathcal V_h$ define a sequence that converges to $u_0$ in $\mathcal V$ as $h \searrow 0$ and let $u_0$ have positive $\mathcal H$-distance $\sigma$ to the relative boundary $\overline{\mathcal M}^\mathsf w\setminus \mathcal M$.
Then there exists a constant $c > 0$ independent of $\sigma$ and a constant $h_0>0$ such that there is a unique  
$u_h$ in $W(0,T^*;\mathcal V,\mathcal H)\cap L_\eta(0,T^*;\mathcal V)$
that solves \Cref{problem discrete} on the time interval $[0,T^*]$ when $T^* < \sigma^2/c$ for all $h\leq h_0$. Furthermore, $u_h$ converges to the unique solution $u$ of \Cref{problem 1} in $W(0,T^*;\mathcal V,\mathcal H)\cap L_\eta(0,T^*;\mathcal V)$ weakly in $L_2(0,T^*;\mathcal V)$ and strongly in $C(0,T^*;\mathcal H)$, while the weak derivatives $u_h'$ converge weakly to $u'$ in $L_2(0,T^*,\mathcal H)$.
\end{theorem}
\begin{proof}
Since $u_{0,h}$ converges to $u_{0}$ in $\mathcal V$, there is an $h_0>0$ such that $\|u_{0,h}-u_{0}\|_\mathcal V\leq \sigma/2$ and $\|u_{0,h}-u_{0}\|_\mathcal H\leq \sigma/2$ for all $h\leq h_0$ due to~\eqref{eq:norminequality}. Furthermore, we can choose $h_0$ small enough, such that $\|u_{0,h}\|_\mathcal V^2\leq 2\|u_{0}\|_\mathcal V^2$ and $\|u_{0,h}\|_\mathcal H^2\leq 2\|u_{0}\|_\mathcal H^2$. 
Therefore, the $\mathcal H$-distance of $u_{0,h}$ from $\overline{\mathcal M}^\mathsf w\setminus \mathcal M$ is at least $\sigma/2$. Hence, applying \Cref{thm:existandunique} with $\mathcal V_h$ in place of $\mathcal V$ provides us with solutions $u_h$ to \Cref{problem discrete} on a time interval $[0,T^*]$ with $T^*<\sigma^2/(4c)$ for every $h\leq h_0$, where $c$ can be chosen as the right-hand side of the following estimate~\eqref{eq:uhprimebound}. \Cref{thm:existandunique} provides us with the estimates 
\begin{align}
    \|u_h\|^2_{L_2(0,T^*;\mathcal V)}&\leq 2\|u_0\|_\mathcal H^2+C_1\|f\|^2_{L_2(0,T^*;\mathcal H)},
    \\
    \|u_h'\|^2_{L_2(0,T^*;\mathcal H)}&\leq C_2\left(2\|u_0\|_\mathcal V^2+\|f\|^2_{L_2(0,T^*;\mathcal H)}\right),
   \label{eq:uhprimebound}  \\
    \|u_h\|^2_{L_\infty(0,T^*;\mathcal V)}&\leq C_3\left(2\|u_0\|_\mathcal V^2+\|f\|^2_{L_2(0,T^*;\mathcal H)}\right).
\end{align}
Note that by \eqref{eq:uhprimebound}, we can assume that for $h$ sufficiently small, $\norm{u_h(t) - u_{0}}_{\mathcal H}\leq \sigma - \delta $ for a $\delta > 0$.  
As a consequence, there is a subsequence $(u_h)$ converging weakly to $\tilde u$ in $L_2(0,T^*;\mathcal V)$ and weakly$^*$ in $L_\infty(0,T^*;\mathcal V)$ and the derivatives $(u'_h)$ converging weakly to $\tilde w$ in $L_2(0,T^*;\mathcal H)$.

We next show that $\tilde w$ is the weak derivative of ${\tilde u}$. For this, we need to verify that
\[
  \int_0^{T^*} \langle \tilde w(t), v \rangle\, \phi(t)+\langle{\tilde u}(t), v\rangle\,\phi'(t)\,d t = 0
\]
for arbitrary $v \in \mathcal V$ and $\phi \in  C_0^\infty (0,T^*)$. 
For any $v_h\in\mathcal V_h$ we may add and subtract the weak derivative of ${u}_h$ to obtain
\begin{multline*}
  \int_{T^*} \langle \tilde w(t), v_h \rangle\,\phi(t)\,d t + \langle{\tilde u}(t), v_h\rangle\,\phi'(t)\,d t \\
  {}={} \int_{T^*} \langle \tilde w(t) - u_h'(t), v_h\rangle \,\phi(t)\,+ \langle{\tilde u}(t) - u_h(t), v_h\rangle\,\phi'(t)\,d t.
\end{multline*}
Now let $(v_h)$ be a sequence converging to $v$ in $\mathcal V$. Then
\begin{multline*}
 \int_{T^*} \langle \tilde w(t), v \rangle\,\phi(t)\,d t + \langle{\tilde u}(t), v\rangle\,\phi'(t)\,d t=
 \lim_{h\searrow 0}
 \int_{T^*} \langle \tilde w(t), v_h \rangle\,\phi(t)\,d t + \langle{\tilde u}(t), v_h\rangle\,\phi'(t)\,d t \\
=
\lim_{h\searrow 0}
\int_{T^*} \langle \tilde w(t) - u_h'(t), v_h\rangle \,\phi(t)\,+ \langle{\tilde u}(t) - u_h(t), v_h\rangle\,\phi'(t)\,d t=0
\end{multline*}
since $v_h\phi$ converges strongly to $v\phi$ in $L_2(0,T^*;\mathcal V)$.
Therefore, the sequence $(u_h)$ converges weakly in $W(0,T^*;\mathcal V,\mathcal H)$ to $\tilde u$. Due to the Aubin-Lions theorem, and by boundedness in $L_\infty(0,T^*;\mathcal V)$, it also converges strongly in $C(0,T^*;\mathcal H)$ to $\tilde u$, and $\tilde u(0)=\lim_{h\searrow 0}u_h(0)=\lim_{h\searrow 0}u_{0,h}=u_0$. 

It remains to show that $\tilde u$ satisfies~\eqref{eq:prob1} and therefore agrees with the unique solution $u$ of \Cref{problem 1} in $W(0,T^*;\mathcal V,\mathcal H)\cap L_\eta(0,T^*;\mathcal V)$ provided by Theorem \ref{thm:existandunique}.
By a subsequence-of-subsequence argument, it then follows  that the entire sequence converges to $\tilde u$. 
Let $Q_h$ be the $\mathcal{V}$-orthogonal projection onto $\mathcal{V}_h$.
For $v\in T_{\tilde u(t)}\mathcal M\cap\mathcal V$, let $v_h = Q_h v$, which converges strongly to $v$ in $\mathcal V$ by Property~\ref{property:discr_approximation1}. This also implies that the sequence is uniformly bounded in $\mathcal H$.
By \eqref{eq:prob_discr}, we have 
\[
\langle u_h'(t), P_{u_h(t)}v_h\rangle + a(u_h(t),P_{u_h(t)}v_h;t) = \langle f(t), P_{u_h(t)}v_h \rangle
\]
for almost every $t$, since $P_{u_h(t)}v_h\in T_{u_h(t)} \mathcal M\cap \mathcal V_h$ by Property~\ref{property:discr_compat2}. We have chosen the time interval such that $u_h(t)\in\mathcal M'\subset \mathcal M$ lie in a weakly compact subset for all $t\in[0,T^*]$. Hence, using Assumption \ref{property:curvature},
\begin{multline}\label{eq:convergenceoftestfunction}
\|v-P_{u_h(t)}v_h\|_\mathcal H
\leq 
\|v-P_{u_h(t)}v\|_\mathcal H+\|P_{u_h(t)}(v-v_h)\|_\mathcal H\\
\leq 
\kappa(\mathcal M')\|\tilde u(t)-u_h(t)\|_\mathcal H\|v\|_\mathcal H+\|(v-v_h)\|_\mathcal H,
\end{multline}
and thus $P_{u_h(t)}v_h$ converges strongly to $v$ in $\mathcal H$. Using a similar argument as in the proof of 
\Cref{thm:existandunique} in~\cite{Bachmayr21},
it suffices to show 
\[
\int_0^{T^*}\langle \tilde u'(t), v(t)\rangle + a(\tilde u(t),v(t);t) - \langle f(t), v(t) \rangle \, d t=0
\]
for all $v\in L_\infty(0,T^*;\mathcal V)$ with $v(t)\in T_{\tilde u(t)}\mathcal M\cap\mathcal V$ for almost every $t$.

Since $P_{u_h(t)}Q_hv(t)$ converges to $v(t)$ in $\mathcal H$ for almost all $t\in[0,T^*]$, and we have the square integrable bound~\eqref{eq:convergenceoftestfunction},
the sequence $P_{u_h(t)}Q_hv(t)$ converges strongly to $v$ in $L_2(0,T^*;\mathcal H)$. This together with weak convergence of $(u'_h)$ in $L_2(0,T^*;\mathcal H)$ implies 
\[
\lim_{h\searrow 0} 
\int_0^{T^*}\!
\langle u_h'(t), P_{u_h(t)}Q_hv(t)\rangle - \langle f(t), P_{u_h(t)}Q_hv(t)\rangle
\, d t
=\int_0^{T^*}\!\langle \tilde u'(t), v(t)\rangle - \langle f(t), v(t) \rangle \, d t.
\]
Finally, we use Assumption~\ref{property:splitting}. We have
\begin{multline}\label{eq:A1A2split}
 a(u_h(t),P_{u_h(t)}Q_hv(t);t)-
a(\tilde u(t),v(t);t)  \\
= 
 \langle A_1(t)u_h(t),P_{u_h(t)}Q_hv(t)\rangle -
\langle A_1(t)\tilde u(t),v(t)\rangle  \\
+
 \langle A_2(t)u_h(t),P_{u_h(t)}Q_hv(t)\rangle -
\langle A_2(t)\tilde u(t),v(t)\rangle 
\end{multline}
and due to Assumption~\ref{property:partA1}
\[
\langle A_1(t)u_h(t),P_{u_h(t)}Q_hv(t)\rangle=\langle A_1(t)u_h(t),Q_hv(t)\rangle. 
\]
This implies
\[
\lim_{h\searrow 0}\int_0^{T^*}\abs{ \langle A_1(t)u_h(t),P_{u_h(t)}Q_hv(t)\rangle -
\langle A_1(t)u(t),v(t)\rangle } \, d t =0
\]
as $u_h$ converges weakly to $\tilde u$  and $Q_hv$ converges strongly to $v$ in $L_2(0,T^*;\mathcal V)$. For the second summand in \eqref{eq:A1A2split} we have
\begin{multline*}
\langle A_2(t)u_h(t),P_{u_h(t)}Q_hv(t)\rangle -
\langle A_2(t)\tilde u(t),v(t)\rangle  \\
= 
 \langle A_2(t)u_h(t),P_{u_h(t)}Q_hv(t)-v(t)\rangle
+
\langle A_2(t)(\tilde u(t)-u_h(t)),v(t)\rangle 
\end{multline*}
where 
\[ 
  \abs{\langle A_2(t)u_h(t),P_{u_h(t)}Q_hv(t)-v(t)\rangle} \leq 
\gamma \|u_h(t)\|_\mathcal V^\eta\|P_{u_h(t)}Q_hv(t)-v(t)\|_\mathcal H  
\]
and $\int_0^{T^*} \|u_h(t)\|_\mathcal V^\eta\|P_{u_h(t)}Q_hv(t)-v(t)\|_\mathcal H  \,d t \to 0$.
Moreover, 
\[ \int_0^{T^*}
\langle A_2(t)(\tilde u(t)-u_h(t)),v(t)\rangle \,d t\to 0 \quad\text{ as $h\searrow 0$,} \] 
since $u_h$ converges weakly to $\tilde u$  in $L_2(0,T^*;\mathcal V)$. Taken together with~\eqref{eq:convergenceoftestfunction} and the uniform bound of $u_h$ in $L_\infty(0,T^*;\mathcal V)$, we have 
\[
\int_0^{T^*} a(u_h(t),P_{u_h(t)}Q_hv(t);t)-
a(\tilde u(t),v(t);t) \,d t\to 0\quad \text{as $h\searrow 0$}
\]
and hence
\begin{multline*}
    \int_0^{T^*}\langle \tilde u'(t), v(t)\rangle + a(\tilde u(t),v(t);t) - \langle f(t), v(t) \rangle \, d t\\
    =
    \lim_{h\searrow 0}
    \int_0^{T^*} \langle u_h'(t), P_{u_h(t)}v_h\rangle + a(u_h(t),P_{u_h(t)}v_h;t) -\langle f(t), P_{u_h(t)}v_h \rangle\, d t
    =
    0
\end{multline*}
for all $v\in L_\infty(0,T^*;\mathcal V)$ with $v(t)\in T_{\tilde u(t)}\mathcal M\cap\mathcal V$ for almost every $t$. 
\end{proof}

\section{Properties of low-rank tensor manifolds in Hilbert space}\label{sec:tensor train model}

In this section we return to our model problem~\eqref{eq:modelproblem1} in its weak formulation~\eqref{eq:modelproblem2} and apply the theory developed above for low-rank models of multivariate functions. In the model problem $\mathcal H=L_2(\Omega)$ and $\mathcal V=H^1_{0}(\Omega)$, the compact embedding $\mathcal V \hookrightarrow \mathcal H$ is due to the Rellich-Kondrachov theorem and~\eqref{eq:norminequality} is the Poincar\'e inequality.

\subsection{Low-rank tensor manifolds in function space}

Let $\Omega = \Omega_1\times\cdots\times\Omega_d$, where  $\Omega_\mu$ is a bounded domain in a Euclidean space for $\mu=1,\ldots,d$. We write $H^\mu = L_2(\Omega_\mu)$ for abbreviation. The space $\mathcal H = L_2(\Omega_1 \times \dots \times \Omega_d) = H^1 \otimes \dots \otimes H^d$ is a tensor product Hilbert space. In DLRA, one considers low-rank manifolds $\mathcal M$ in such spaces. We consider manifolds of the  general form
\begin{align}
\mathcal M = \Big\{ u = \sum_{k_1 = 1}^{r_1} \cdots \sum_{k_d = 1}^{r_d} C(k_1,\dots,k_d)& \, u^1_{k_1} \otimes \dots \otimes u^d_{k_d} \colon \label{eq: Tucker format} \\ 
& C \in \mathcal{M}_\mathrm{c} \subset \R^{r_1 \times \dots \times r_d}_*,
\text{$G(u^\mu)$ is invertible}
\Big\}. \notag
\end{align}
Here $\R^{r_1 \times \dots \times r_d}_*$ denotes the dense and open subset of ``regular'' $r_1\times \dots \times r_d$ tensors with full multilinear rank $(r_1,\dots,r_d)$, and $G(u^\mu)= [\langle u^\mu_i, u^\mu_j \rangle]_{ij} \in \R^{r_\mu \times r_\mu}$ is the Gramian of the system $\{u^\mu_1,\dots,u^\mu_{r_\mu} \}$. We assume that $\mathcal{M}_\mathrm{c}$ is a smooth submanifold in $\R^{r_1 \times \dots \times r_d}_*$ that we additionally assume to be invariant under changes of basis in the sense that for all $C \in \mathcal{M}_\mathrm{c}$,
\begin{equation}\label{eq: invariance Mhat}
C \times_1 T^1 \times_2 \dots \times_d T^d \in \mathcal{M}_\mathrm{c}
\quad\text{for all invertible matrices $T^1,\dots,T^d$.}
\end{equation}
Here we use the notation $\times_\mu$ for left multiplication of a matrix onto the $\mu$-th mode a tensor ~\cite{KoldaBader}.

The definition of $\mathcal M$ results in a constrained version of the Tucker format (for which $\mathcal{M}_\mathrm{c} = \R^{r_1 \times \dots \times r_d}_*$), and covers continuous versions of general tree based tensor formats, for example by letting $\mathcal{M}_\mathrm{c}$ be a corresponding finite-dimensional low-rank tensor manifold in $\R^{r_1 \times \dots \times r_d}$, such as manifolds of tensor trains \cite{Holtz12,UschmajewVandereycken:20} or of hierarchical Tucker tensors \cite{UschmajewVandereycken:13,Bachmayr16} with fixed ranks.

Note that it follows from the assumed properties of $\mathcal M_{\mathrm c}$, that this set is not closed in $\R^{r_1 \times \dots \times r_d}$, and hence $\overline{\mathcal M_{\mathrm c}} \setminus \mathcal M_{\mathrm c} \neq \emptyset$. To see this, let $C \in \mathcal M_{\mathrm c}$ and $T^1_n,\dots,T^d_n$ be invertible matrices converging to $T_*^1,\dots,T_*^d$ such that at least one of the limits is not invertible. Then $C \times_1 T^1_n \times_2 \dots \times_d T^d_n \in \mathcal M_{\mathrm c}$ for all $n$ by~\eqref{eq: invariance Mhat}, but the limit $C \times_1 T^1_* \times_2 \dots \times_d T^d_*$ does not have full multilinear rank, and hence is not in $\mathcal M_{\mathrm c}$. This example also shows that the set $\mathcal M$ is not closed in $\mathcal H$, and in particular not weakly closed.
Moreover, in \Cref{lm: closureManifold} we will see that the closure and weak closure of $\mathcal M$ coincide with the set described in~\eqref{eq: Tucker format} with~$\mathcal M_\mathrm c$ replaced by~$\overline{\mathcal M_\mathrm c}$.

For investigating the manifold properties of $\mathcal M$, the concepts of matricizations and minimal subspaces play a crucial role. For every $\mu=1,\dots,d$, we can identify $u$ with an element $M_u^\mu$, called the $\mu$-th matricization of $u$, in the subspace $H^\mu \otimes H^{\neq \mu}$, where $H^{\neq \mu} = \bigotimes_{\nu \neq \mu} H^\nu$. Assuming $u \in \mathcal M$ as above and letting
\begin{equation}\label{eq: definition of v_k}
v_{k_\mu}^\mu = \sum_{k_1 =1}^{r_1} \cdots \sum_{k_{\mu-1} =1}^{r_{\mu-1}} \sum_{k_{\mu+1} =1}^{r_{\mu+1}} \cdots \sum_{k_d =1}^{r_d} C(k_1,\dots,k_d) \, u^1_{k_1} \otimes \dots \otimes u^{\mu-1}_{k_{\mu-1}} \otimes u^{\mu+1}_{k_{\mu+1}} \otimes \dots \otimes u^d_{k_d},
\end{equation}
one has
\begin{equation}\label{eq: representation T_u}
M_u^\mu = \sum_{k_\mu=1}^{r_\mu} u_{k_\mu}^\mu \otimes v_{k_\mu}^\mu.
\end{equation}
Since the core tensor $C$ has full multilinear rank by~\eqref{eq: Tucker format}, one can show that the $v_{k_\mu}^\mu$ are also linearly independent. Now define
\[
\mathcal U^\mu = \spn\{u^\mu_{1},\dots,u^\mu_{r_\mu} \}
\]
and
\[
\mathcal V^\mu = \spn\{ v^\mu_1,\dots,v^\mu_{r_\mu} \},
\]
then~\eqref{eq: representation T_u} expresses the fact, that $M_u^\mu$ is an element of the ``matrix subspace'' $\mathcal U^\mu \otimes \mathcal V^\mu$ and $\rank(M_u^\mu) = r_\mu$. We call $\mathcal U^\mu$ the $\mu$-th minimal subspace of $u \in \mathcal M$.

Choosing an orthonormal basis for each space $H^\mu = L_2(\Omega_\mu)$, we obtain an isomorphism between $H^\mu$ and $\ell_2(\N)$ for each $\mu$. This in turn defines a tensor space isomorphism between $\mathcal{H}$ and $\ell_2(\mathbb N) \otimes \dots \otimes \ell_2(\mathbb N)$. 
In what follows, in order to use a more common matrix and tensor notation, we can thus assume without loss of generality that
\[
  \mathcal H = \ell_2(\mathbb N^d) = \ell_2(\mathbb N) \otimes \dots \otimes \ell_2(\mathbb N),
\]
and thus consider $\mathcal M$ as a set in the tensor product Hilbert space of square summable infinite arrays. The definition of $\mathcal M$ remains the same as in~\eqref{eq: Tucker format}, only that now $u^\mu_{k_\mu} \in \ell_2(\mathbb N)$. 

We will, however, denote the elements of $\ell_2(\mathbb N^d)$ as $X$ instead of $u$, in order to clearly distinguish these sequences from functions. The corresponding matricizations are $M^\mu_X \in \ell_2(\mathbb N) \otimes \ell_2(\mathbb N^{d-1})$. The Tucker format~\eqref{eq: Tucker format} can then be written in the usual abbreviated form
\[
X = C \times_1 U^1 \times_2 \cdots \times_d U^d
\]
where $U^\mu = [u_1^\mu,\dots,u^\mu_{r_\mu}] \in (\ell_2(\mathbb N))^{r_\mu}$ contains a basis for $\mathcal U^\mu$. Here the multiplications $\times_\mu$ are defined as for finite tensors. 

\subsubsection{Manifold structure}

Using the concept of manifolds in Banach space as presented in~\cite[Ch.~73]{Zeidler88} we can prove the following result.

\begin{theorem}\label{thm: manifold}
Let $X = C \times_1 U^1 \times_2 \cdots \times_d U^d$ be in $\mathcal M$ defined as in \eqref{eq: Tucker format} satisfying \eqref{eq: invariance Mhat}. Then the following statements hold.

\begin{enumerate}[label={\upshape (\roman*)},leftmargin=*,itemsep = 10pt]
\item
There exists an open neighborhood $\mathcal O$ of $X$ and a submersion $g$ defined on $\mathcal O$ such that $\mathcal M \cap \mathcal O = g^{-1}(0)$. Consequently,  $\mathcal M \cap \mathcal O$ is a smooth submanifold in the Hilbert space $\mathcal H$. The tangent space $T_{X} \mathcal M$ at $X \in \mathcal M \cap \mathcal O$ is the null space of $g'(X)$.
\item
There exists a continuously Fr\'echet-differentiable homeomorphism $\varphi$ from a neighborhood of zero in $T_{X_*} \Mc$ to $\mathcal M \cap \mathcal O$ satisfying $\varphi(\xi) = X_* + \xi + o(\| \xi \|_\mathcal{H})$ for all $\xi$ in that neighborhood. Moreover, $\varphi$ is also an immersion and hence a local embedding for~$\mathcal M$.
\item
The tangent space equals the subspace spanned by elements of the form
\begin{equation}\label{eq: tangent vectors}
\xi = \dot{C} \times_1 U^1 \times_2 \dots \times_d U^d + C \times_1 \dot{U}^1 \times_2 \dots \times_d U^d + \dots + C \times_1 U^1 \times_2 \dots \times_d \dot{U}^d
\end{equation}
with $\dot C \in T_C \mathcal{M}_\mathrm{c}$ and $(U^\mu)^\T \dot U^\mu = 0_{r_\mu \times r_\mu}$ for $\mu = 1,\dots,d$ (that is, the columns of $\dot U^\mu$ span a subspace orthogonal to $\mathcal U^\mu$).
\end{enumerate}
\end{theorem}

The proof of this theorem is given in the appendix. For an alternative treatment of low-rank manifolds in Banach spaces, see \cite{Falco19,Falco21}.

\subsubsection{Tangent space projection}

We now consider the orthogonal projection onto the tangent space $T_X \mathcal M$ at given $X = C \times_1 U^1 \times_2 \dots \times_d U^d$. By Theorem~\ref{thm: manifold}(iii), $T_X \mathcal M$ is spanned by elements $\xi = \xi_0 + \xi_1 + \dots + \xi_d$ of the form~\eqref{eq: tangent vectors}. Here the elements $\xi_0 = \dot{C} \times_1 U^1 \times_2 \dots \times_d U^d$ with $\dot C \in T_{C} \mathcal{M}_\mathrm{c}$ span a subspace of $\mathcal U^1 \otimes \dots \otimes \mathcal U^d$ which we denote by $\mathcal S_X$. For $\mu=1,\dots,d$, the elements $\xi_\mu = C \times_1 U^1 \times_2 \dots \times_{\mu} \dot U^\mu \times_{\mu+1} \dots \times_d U^d$ with $(U^\mu)^\T \dot U^\mu = 0$ are equivalently described via their matricization as
\[
M_{\xi_\mu}^\mu = \dot U^\mu (V^\mu)^\T = \sum_{k_\mu=1}^{r_\mu} \dot u^\mu_{k_\mu} \otimes v^\mu_{k_\mu}
\]
due to the definition~\eqref{eq: definition of v_k} of $V^\mu$. Since actually any element in the space $(\mathcal U^\mu)^\perp \otimes \mathcal V^\mu$ can be written in this way, the $M_{\xi_\mu}^\mu$ span this space. Treating the $(\mathcal U^\mu)^\perp \otimes \mathcal V^\mu$ as subspaces of $\ell_2(\mathbb N^d)$ (in a slight abuse of notation) we conclude that
\begin{equation}\label{eq: decomposition of tangent space}
T_X \mathcal M = \mathcal S_X \oplus [(\mathcal U^1)^\perp \otimes \mathcal V^1] \oplus \dots \oplus [(\mathcal U^d)^\perp \otimes \mathcal V^d],
\end{equation}
which indeed is an orthogonal decomposition as can be seen from the fact that $\mathcal V^\mu \subseteq \mathcal U^1 \otimes \dots \otimes \mathcal U^{\mu-1} \otimes \mathcal U^{\mu+1} \otimes \dots \otimes \mathcal U^d$ for every $\mu$.

In the following proposition, we compute the tangent space projection under the assumption that the matrices $U^\mu$ have orthonormal columns; see Remark~\ref{remark: tangentprojection} for the general formula.

\begin{proposition}\label{prop: tangent space projection}
Let $X = C \times_1 U_1 \times_2 \dots \times_d U_d \in \mathcal M$, and assume $(U^\mu)^\T U^\mu = \id$. The orthogonal projection onto the tangent space $T_X \mathcal M$ is given as
\begin{equation}\label{eq: projector splitting}
P_X = P^0_X + P^1_X + \dots + P^d_X
\end{equation}
with $P^1_X,\dots,P^d_X$ being implicitly defined via their action on matricizations as
\begin{equation}\label{eq: matricization projectors}
M^\mu_{P^\mu_X(Z)} = (I - P_{\mathcal U^\mu}) M^\mu_Z P_{\mathcal V^\mu}, \quad \mu = 1,\dots,d.
\end{equation}
The projector $P_X^0$ is defined as
\begin{equation}\label{eq: projector P0}
P_X^0 (Z) = P_C(C_Z)\times_1 U^1 \times_2 \dots \times_d U^d,
\end{equation}
where $P_C$ is the orthogonal tangent space projector onto $T_C \mathcal{M}_\mathrm{c}$ in $\R^{r_1 \times \dots \times r_d}$, and $C_Z = Z \times_1 (U^1)^T \times_2 \dots \times_d (U^d)^T$.
\end{proposition}

\begin{proof}
By~\eqref{eq: decomposition of tangent space}, the single terms $\xi_0,\dots,\xi_d$ in the tangent vector representation~\eqref{eq: tangent vectors} belong to mutually orthogonal subspaces. Therefore, the orthogonal projection $P_X(Z) = \xi_0 + \xi_1 + \dots + \xi_d$ onto $T_X \mathcal M$ can be decomposed accordingly as in~\eqref{eq: projector splitting}. Here the $P_X^\mu$ for $\mu = 1,\dots,d$ are the projections on $(\mathcal U^\mu)^\perp \otimes \mathcal V^\mu$ which have the asserted form. We consider the projection $\xi_0 = P_X^0(Z)$ of a given $Z$ onto the space $\mathcal S_X$ in~\eqref{eq: decomposition of tangent space}. We write
\[
\xi_0 = K \times_1 U^1 \times_2 \dots \times_d U^d
\]
and need to determine $K \in T_C \mathcal{M}_\mathrm{c}$. The orthogonality condition for the projection is
\[
 0 = \langle Z - \xi_0, \dot C \times_1 U^1 \times_2 \dots \times_d U^d \rangle
\] 
for all $\dot C \in T_C \mathcal{M}_\mathrm{c}$. Using properties of tensor-matrix multiplication, we rewrite this as
 \begin{align*}
 0 &= \langle Z \times_1 (U^1)^\T \times_2 \dots \times_d (U^d)^\T - \xi_0 \times_1 (U^1)^\T \times_2 \dots \times_d (U^d)^\T, \dot C \rangle \\
 &= \langle Z \times_1 (U^1)^\T \times_2 \dots \times_d (U^d)^\T - K, \dot C \rangle
= \langle  C_Z - K, \dot C \rangle.
\end{align*}
Since this holds for all $\dot C \in T_C \mathcal{M}_\mathrm{c}$, it follows that $K$ equals the orthogonal projection of $C_Z$ onto $T_C \mathcal{M}_\mathrm{c}$.
\end{proof}

\begin{remark}\label{remark: tangentprojection}
If the $U^\mu$ are not orthonormal, then the formula for $P_X^0 (Z)$ needs to be adjusted to
\[
P_X^0 (Z) = \Pi_C(C_Z),
\]
where $\Pi_C$ is the orthogonal projection in $\R^{r_1 \times \dots\times r_d}$ onto $T_C \mathcal{M}_\mathrm{c}$ with respect to the inner product induced by the operator $\mathbf A = (U^1)^T U^1 \otimes \dots \otimes (U^d)^T U^d$, which is symmetric and positive definite on $\R^{r_1 \times \dots\times r_d}$. This projection is given by 
\(
\Pi_C = (P_C \mathbf A P_C)^{-1} P_C \mathbf A,
\) 
where $(P_C \mathbf A P_C)^{-1}$ denotes the inverse of $P_C \mathbf A P_C$ on $T_C \mathcal{M}_\mathrm{c}$. 
\end{remark}

\subsubsection{Distance to boundary}

As we will see in \Cref{sec:curvature},
curvature estimates as in ~\ref{property:curvature} for low-rank tensor manifolds can be expressed in terms of inverses of smallest singular values of certain matricizations.
In this subsection, we therefore estimate the smallest singular values of matricizations of a tensor $X\in\mathcal M$ from below by its distance to the boundary $\overline{\mathcal M}^w\setminus \mathcal M$.
This will have the effect, that on every weakly compact subset~$\mathcal M'\subseteq \mathcal M$ these singular values remain bounded from below.

We first give a characterization of the weak closure of $\mathcal M$.
In what follows, by $\norm{\cdot}$ without further specification we denote the Frobenius norm of tensors. 

\begin{lemma}\label{lm: closureManifold}
  Let $\mathcal M$ be of the form~\eqref{eq: Tucker format} with $\mathcal M_\mathrm c$ satisfying~\eqref{eq: invariance Mhat}. Then 
\[
    \overline{\mathcal M}^w = \overline{\mathcal M} = \Big\{ X = C \times_1 U^1 \times_2 \cdots \times_d U^d\colon
    C\in \overline{\mathcal M_\mathrm{c}},\quad (U^\mu)^\T U^\mu\in \mathrm{GL_{r_\mu}}
    \Big\}, 
\]
    that is, the weak closure and closure of~$\mathcal M$ coincide and are of the form~\eqref{eq: Tucker format} with $\mathcal M_\mathrm c$ replaced by $\overline{\mathcal{M}_\mathrm{c}}$.
\end{lemma}
\begin{proof}
  Let $(X_n)\subset \mathcal M$ be a sequence converging weakly to $X\in\overline{\mathcal M}^w$. By~\cite[Thm.~6.29]{Hackbusch19}, there are $r_\mu$~dimensional subspaces $\mathcal U^\mu$ such that $X\in \mathcal U^1\otimes \dots\otimes \mathcal U^d$. In particular, let $U^\mu\in (\ell_2(\mathbb N))^{r_\mu}$ be orthonormal bases of the spaces $\mathcal U^\mu$. Then
  \[
  X = \left(X \times_1 (U^1)^\T \times_2 \cdots \times_d (U^d)^\T\right)  \times_1 U^1 \times_2 \cdots \times_d U^d
    = C \times_1 U^1 \times_2 \cdots \times_d U^d
  \]
  with $C\in \R^{r_1\times\dots}\times r_d$.
  Moreover, since $X_n\in \mathcal M$, we have
  \[
    X_n = C_n  \times_1 U^1_n \times_2 \cdots \times_d U^d_n
  \]
  and by weak convergence, we have
  \[
    \lim_{n\to\infty} C_n  \times_1 (U^1)^\T U^1_n \times_2 \cdots \times_d (U^d)^\T U^d_n  
    =
    C.
  \]
  By the invariance condition~\eqref{eq: invariance Mhat} for $\mathcal M_\mathrm c$, it follows that $C\in \overline{\mathcal M_\mathrm c}$. This shows
  \[
    \overline{\mathcal M}^w = \Big\{ X = C \times_1 U^1 \times_2 \cdots \times_d U^d\colon
    C\in \overline{\mathcal M_\mathrm{c}},\quad (U^\mu)^\T U^\mu\in \mathrm{GL_{r_\mu}}
    \Big\}.
  \]
  Now let $(\tilde C_n) \subset \mathcal M_\mathrm c$ be a sequence converging to $C$. Then the sequence defined by
  \[
    \tilde X_n = \tilde C_n  \times_1 U^1 \times_2 \cdots \times_d U^d
  \]
  converges strongly to $X$ as $\norm{\tilde X_n-X}_{\ell_2(\N^d)} = \norm{\tilde C_n -C}$. This proves the assertion.
\end{proof}

For $X \in \mathcal{M}$ and $\mu=1,\ldots,d$, let $\{ u^\mu_1,\ldots, u^\mu_{r_\mu} \}$ be the left singular vectors of the matricization $M^\mu_X$. Then 
\[
   M^\mu_X = \sum_{k=1}^{r_\mu} u^\mu_k \otimes v^\mu_k  
\]
with $\{ v^\mu_1,\ldots, v^\mu_{r_\mu}\}$ orthogonal in $\ell_2(\N^{d-1})$ such that
\begin{equation}\label{eq:singvalX}
  \sigma^\mu_k = \sigma^\mu_k(X) = \norm{ v^\mu_k }_{\ell_2(\N^{d-1})}, \quad k = 1,\ldots,r_\mu,
\end{equation}
are the singular values of $M^\mu_X$, for which we may assume
\[
   \sigma^\mu_1 \geq \sigma^\mu_2 \geq \ldots \geq \sigma^\mu_{r_\mu}.  
\]
Further, we define 
\begin{equation}\label{eq: distance}
\sigma = \dist(X,\overline{{\mathcal M}}^w \setminus {\mathcal M}).
\end{equation}
\begin{proposition}\label{prop:singvaldistest}
  
  Let $\sigma^\mu_k$ for $\mu=1,\ldots,d$ and $k=1,\ldots,r_\mu$ be defined as above.  Then
  \[
     \min_{\mu \in \{1,\ldots,d\} } \sigma^\mu_{r_\mu} \geq  \sigma.
  \]
\end{proposition}

\begin{proof}
Let $\mu \in \{1,\ldots,d\}$ and let $u^\mu_1, \ldots, u^\mu_{r_\mu}$ be the left singular vectors of $M^\mu_X$ associated to $\sigma^\mu_1,\ldots,\sigma^\mu_{r_\mu}$, respectively. Then for the tensor $\tilde X$ defined by its matricization 
\[
   M^\mu_{\tilde X} =  \sum_{k=1}^{r_\mu - 1} u^\mu_k (u^\mu_k)^\top M^\mu_X,
\]
we have $\norm{X - \tilde X}_{\ell_2(\N^d)} = \sigma^\mu_{r_\mu}$, and $\tilde X \notin  \mathcal M$.
Furthermore, we have
\[
 X = C \times_1 U^1 \times _2 \dots \times_d U^d, 
\]
with $U^\mu = [u^\mu_1,\dots, u^\mu_{r_\mu}]$ and
\[
  \tilde X = C \times_1 U^1 \times _2\dots 
  \times_\mu   U^\mu (\tilde U^\mu)^\top U^\mu 
  \times_{\mu+1}\dots \times_d U^d\\
  = (C\times_\mu  (\tilde U^\mu)^\top U^\mu ) \times_1 U^1 \times _2 \dots \times_d U^d
\]
with $\tilde U^\mu = [u^\mu_1,\dots, u^\mu_{r_\mu-1},0]$. It follows from \Cref{lm: closureManifold} that $\tilde X\in \overline {\mathcal M}$ and the claim is proven.
\end{proof}

It is important to note that
the distance~$\sigma$ defined in~\eqref{eq: distance} can be expressed as the distance of the core tensor $C$ to the relative boundary of~$\mathcal M_\mathrm c$.

\begin{proposition}\label{prop:distequal}
  Let $X = C \times_1 U_1 \times_2 \dots \times_d U_d$  with $C\in\mathcal{M}_\mathrm{c}$ and orthonormal $U_1,\ldots, U_d$. Then $\dist(C,\overline{\mathcal{M}_\mathrm{c}} \setminus \mathcal{M}_\mathrm{c})= \dist(X,\overline{{\mathcal M}}^w \setminus {\mathcal M})$.
\end{proposition}
\begin{proof}
  First, let $Y\in \overline{{\mathcal M}}^w \setminus {\mathcal M}$ satisfy $\|X-Y\|_{\ell_2(\N^d)}= \dist(X,\overline{{\mathcal M}}^w \setminus {\mathcal M})$. Then the tensor $D = Y\times_1 U_1^\T \times_2 \dots \times_d U_d^\T$  satisfies $\|C-D\|\leq\|X-Y\|_{\ell_2(\N^d)}$ and by \Cref{lm: closureManifold}, we also have $D\in \overline{\mathcal{M}_\mathrm{c}} \setminus \mathcal{M}_\mathrm{c}$, and hence $\dist(C,\overline{\mathcal{M}_\mathrm{c}} \setminus \mathcal{M}_\mathrm{c})\leq \dist(X,\overline{{\mathcal M}}^w \setminus {\mathcal M})$. To show equality, we consider a $D\in\overline{\mathcal{M}_\mathrm{c}} \setminus \mathcal{M}_\mathrm{c}$ with 
  $\|C-D\| = \dist(C,\overline{\mathcal{M}_\mathrm{c}} \setminus \mathcal{M}_\mathrm{c})$. Set $Y = D \times_1 U_1 \times_2 \dots \times_d U_d\in \overline{{\mathcal M}}^w \setminus {\mathcal M}$. Then $\|X-Y\|_{\ell_2(\N^d)}=\|C-D\|$ holds, and thus $\dist(C,\overline{\mathcal{M}_\mathrm{c}} \setminus \mathcal{M}_\mathrm{c})\geq \dist(X,\overline{{\mathcal M}}^w \setminus {\mathcal M})$.
\end{proof}

\subsubsection{Curvature estimates} \label{sec:curvature}

We now turn to the curvature bounds in Assumption~\ref{property:curvature}.
We first derive an estimate for the norm difference $\| P_X - P_Y \|_{\ell_2(\N^d) \to \ell_2(\N^d)}$ of two such projections in operator norm, which can be regarded as a curvature estimate for the manifold $\mathcal M$. It will be required for Assumption~\ref{property:curvature}.

\begin{proposition}\label{prop: curvature estimate}
Assume a curvature estimate
\begin{equation}\label{eq:Mccurvest}
\max_{\norm{Z}=1} \| (P_C - P_{\tilde C})Z \| \le \frac{c}{\hat\sigma} \| C - \tilde C \|\quad\text{for all $C,\tilde C\in \mathcal{M}_\mathrm{c}$,}
\end{equation}
where $\hat\sigma = \dist(C,\overline{\mathcal{M}_\mathrm{c}} \setminus \mathcal{M}_\mathrm{c})$ and where $c>0$ is independent of $C, \tilde C$.
Let $X,Y \in \mathcal M$ with corresponding tangent space projections $P_X$ and~$P_Y$. Then
\[
\| P_X - P_Y \|_{\ell_2(\N^d) \to \ell_2(\N^d)} \le \left( \frac{\sqrt{2}c}{\sigma} +  2(\sqrt{2}+1)\sum_{\mu=1}^d \frac{1}{\sigma^\mu_{r_\mu}}  \right) \| X - Y \|_{\ell_2(\N^d)}
\]
where $\sigma = \dist(X,\overline{\mathcal M}^w \setminus \mathcal M)$ and $\sigma^\mu_{r_\mu}$ is the smallest singular value of the $\mu$-th matricization.
\end{proposition}
  Note that $\sigma^\mu_{r_\mu}\geq \sigma$ for each $\mu$ as a consequence of Proposition \ref{prop:singvaldistest}. Therefore, we have the simpler estimate 
  \[
  \| P_X - P_Y \|_{\ell_2(\N^d) \to \ell_2(\N^d)} \le \frac{\sqrt{2}c + 2d(\sqrt{2}+1)}{\sigma} \| X - Y \|_{\ell_2(\N^d)}.
  \]
Since on every weakly compact subset~$\mathcal M'\subseteq \mathcal M$ the distance~$\sigma$ to the boundary is bounded from below (recall that $\mathcal{M}$ itself is not weakly closed), we obtain the first curvature estimate in~\ref{property:curvature}.
In the proof of \Cref{prop: curvature estimate}, we use the following lemma.

\begin{lemma}\label{lem: difference of projections}\label{lm: basis inequality}
  Let $U,V \in [\ell_2(\N)]^r$ be orthonormal (that is, $U^\T U = V^\T V = \id$) such that the $r \times r$ matrix $U^\T V$ is symmetric and positive semidefinite.  
  \begin{itemize}
    \item[\upshape{(i)}]
  The corresponding subspace projections $P_{\mathcal U} = U U^\T$ and $P_{\mathcal V} = V V^\T$ satisfy
  \[
  \| U - V \|_{\R^r\to \ell_2(\N)} \le \sqrt{2} \| P_{\mathcal U} - P_{\mathcal V} \|_{\ell_2(\N)\to \ell_2(\N)}.
  \]
  \item[\upshape{(ii)}]
  For all $x,y\in \R^r$, 
  $\|x-y\|\leq\sqrt 2\|Ux-Vy\|_{\ell_2(\N)}$.
  \end{itemize}
\end{lemma}
  \begin{proof}
    After orthogonal change of basis, we may assume the matrix $U^\T V = \Sigma$ to be diagonal with entries $1\geq \sigma_i\geq 0$, that is $u_j^\T v_i = \sigma_i \delta_{ij}$.

    Ad (i).
    We define the spaces $W_i = \spa\{u_i,v_i\}$ for $i=1,\ldots, r$. These are pairwise orthogonal. Furthermore, let $W_{r+1} = \left(\bigoplus_{i=1}^r W_i \right)^\perp$. Then the difference of projections is block-diagonal with respect to the spaces $W_i$, that is, $(P_{\mathcal U} -P_{\mathcal V}) (x u_i + y v_i)= (x+\sigma_i y) u_i - (y+\sigma_i x) v_i$. Therefore, the operator norm is given by
    \[
      \|P_{\mathcal U}-P_{\mathcal V}\|_{\ell_2(\N)\to \ell_2(\N)} = \max_i \max_{x\neq 0\neq y}\frac{1}{\|x u_i + y v_i\|_{\ell_2(\N)}}
      \|(P_{\mathcal U} -P_{\mathcal V}) (x u_i + y v_i)\|_{\ell_2(\N)}.
    \]
    We note the norm equality $\|x u_i + y v_i\|^2_{{\ell_2(\N)}} = x^2 +y^2 +2 \sigma_i xy$. Then on the one hand,
    \begin{align*}
      \|(P_{\mathcal U} -P_{\mathcal V}) (x u_i + y v_i)\|^2_{\ell_2(\N)}
      &=
      (x+\sigma_i y )^2 + (y+\sigma_i x)^2 - 2 \sigma_i(x+\sigma_i y)(y+\sigma_i x)\\
      &=
      x^2+y^2 +2\sigma_i xy -\sigma_i^2 (x^2+y^2+2\sigma_ixy)\\
      &=
      (1-\sigma_i^2)\|x u_i + y v_i\|^2_{\ell_2(\N)}\\
      & \geq (1- \sigma_i) \|x u_i + y v_i\|^2_{\ell_2(\N)},
    \end{align*}
that is, $\|P_{\mathcal U}-P_{\mathcal V}\|^2_{\ell_2(\N)\to \ell_2(\N)}\geq 1- \sigma_i$. On the other hand, we have 
\[
    (U-V)^\T(U-V)  = 2 (\id_r -\Sigma),
\]
and thus $\max_{\norm{w}=1}\| (U - V)w \|^2_{\ell_2(\N)} = 2 \max_i (1-\sigma_i)$, which leads to the desired inequality.

Ad (ii).
 Using inequality $(a-b)^2 \leq 2 a^2 +2b^2$ componentwise, we get
  \begin{align*}
    \|x-y\|^2 
    &= (x-y)^\T \Sigma (x-y) + (x-y)^\T (\id_r-\Sigma )(x-y)\\
    &\leq (x-y)^\T \Sigma (x-y) + 2x^\T (\id_r-\Sigma )x + 2y^\T (\id_r-\Sigma )y\\
    &\leq 2\|x\|^2+2\|y\|^2- 4 x^\T\Sigma y\\
    &= 2 \|Ux\|_{\ell_2(\N)} + 2 \|Vy\|_{\ell_2(\N)} - 4 (Ux)^\T (Vy) = 2\|Ux-Vy\|^2_{\ell_2(\N)},
  \end{align*}
  which is the claim.
\end{proof}

\begin{proof}[Proof of Proposition~\ref{prop: curvature estimate}]
Assume representations 
\[ X = C \times_1 U_1 \times_2 \dots \times_d U_d, \quad Y = \tilde C \times_1 \tilde U_1 \times_2 \dots \times_d \tilde U_d 
\] 
as in Proposition~\ref{prop: tangent space projection}. By using polar decompositions $(U^\mu)^\T \tilde U^\mu = Q^\mu S^\mu$, where $Q^\mu$ is orthogonal and $S^\mu$ is positive semidefinite, we can replace the $U^\mu$ with $U^\mu Q^\mu$ and the core tensor $C$ accordingly such that $(U^\mu)^\T \tilde U^\mu$ is positive semidefinite, which we assume to be the case for all $\mu = 1,\dots,d$.
By Proposition~\ref{prop: tangent space projection},
\[
P_X - P_Y = P_X^0 - P_Y^0 + \sum_{\mu=1}^d P_X^\mu - P_Y^\mu.
\]
We will estimate the single differences separately. Applying the triangle inequality will then prove the assertion.

We first consider any of the projector differences $P_X^\mu - P_Y^\mu$ for $\mu = 1,\dots,d$. By~\eqref{eq: matricization projectors}, they can be written in the $\mu$-th matricization space as
\begin{align*}
P_X^\mu - P_Y^\mu &= (\id - P_{\mathcal U^\mu}) \otimes P_{\mathcal V^\mu} - (\id - P_{\tilde{\mathcal U}^\mu}) \otimes P_{\tilde{\mathcal V}^\mu} \\
&= (\id - P_{\mathcal U^\mu}) \otimes (P_{\mathcal V^\mu} - P_{\tilde{\mathcal V}^\mu}) + (P_{\tilde{\mathcal U}^\mu} - P_{\mathcal U^\mu}) \otimes P_{\tilde{\mathcal V}^\mu}.
\end{align*}
We have
\begin{equation}\label{eq: projector difference}
\| P_{\tilde{\mathcal U}^\mu} - P_{\mathcal U^\mu} \|_{\ell_2(\N) \to \ell_2(\N)} \le \frac{1}{\sigma_{r_\mu}^\mu} \| M_X^\mu - M_Y^\mu \|_{\ell_2(\N^{d-1})\to \ell_2(\N)} \le \frac{1}{\sigma_{r_\mu}^\mu} \| X - Y \|_{\ell_2(\N^d)},
\end{equation}
where again $M_X^\mu$ and $M_Y^\mu$ denote the matricizations of $X$ and $Y$ and $\sigma^\mu_{r_\mu}= \sigma^\mu_{r_\mu}(X)$ denotes the smallest positive singular value of $M_X^\mu$ as in \eqref{eq:singvalX}.  For the first inequality see, for example, the proof of~\cite[Lemma~A.2]{Bachmayr21}, the second one is trivial. The same upper bound holds for $\| P_{\tilde{\mathcal V}^\mu} - P_{\mathcal V^\mu} \|_{\ell_2(\N^{d-1}) \to \ell_2(\N^{d-1})}$. 
 Thus we conclude
\[
\| P_X^\mu - P_Y^\mu \|_{\ell_2(\N^d) \to \ell_2(\N^d)} \le \frac{2}{\sigma} \| X - Y \|_{\ell_2(\N^d)}\,.
\]

We now proceed with estimating the operator norm of the difference $P_X^0 - P_Y^0$. By~\eqref{eq: projector P0},
\begin{align*}
(P_X^0 - P_Y^0)(Z) &= P_C(C_Z)\times_1 U^1 \times_2 \dots \times_d U^d - P_{\tilde C}(\tilde C_Z)\times_1 \tilde U^1 \times_2 \dots \times_d \tilde U^d \\
&= [P_C(C_Z) - P_{\tilde C}(\tilde C_Z)] \times_1 \tilde U^1 \times_2 \dots \times_d \tilde U^d \\
&\qquad {}+{} P_{C}(C_Z) \times_1 [U^1 - \tilde U^1] \times_2 \tilde U^2 \times_3  \cdots \times_d \tilde U^d \\
&\qquad \, \, \, {}\vdots{} \\
&\qquad {}+{} P_{C}(C_Z) \times_1 U^1 \times_2 \dots \times_{d-1} U^{d-1} \times_d [U^d - \tilde U^d]
\end{align*}
where $C_Z = Z \times_1 (U^1)^T \times_2 \dots \times_d (U^d)^T$ and similar for ${\tilde C}_Z$. The first term in the right sum is bounded by
\[
\| P_C(C_Z) - P_{\tilde C}(\tilde C_Z) \|_{\ell_2(\N^d)} 
\le 
\| (P_C - P_{\tilde C})C_Z \| + \| C_Z - \tilde C_Z \|,
\] 
since $U^\mu$ and $\tilde U^\mu$ have orthonormal columns and hence spectral norm one. For the other terms we use \Cref{lem: difference of projections}\upshape{(i)} and~\eqref{eq: projector difference}, which leads to
\begin{multline*}
\| P_{C}(C_Z) \times_1 [U^1 - \tilde U^1] \times_2 \tilde U^2 \times_3  \cdots \times_d \tilde U^d \|_{\ell_2(\N^d)} \le \sqrt{2} \|  P_{\mathcal U^1} - P_{\tilde{\mathcal U}^1} \|_{\ell_2(\N) \to \ell_2(\N)}   \| P_{C}(C_Z) \| \\
\le \frac{\sqrt{2}}{\sigma_{r^1_1}(X)} \|X-Y\|_{\ell_2(\N^d)}\| C_Z \| \le \frac{\sqrt{2}}{\sigma_{r_1}^1} \|X-Y\|_{\ell_2(\N^d)}\| Z \|_{\ell_2(\N^d)} 
\end{multline*}
and we proceed similarly for the further modes. So far we have shown
\begin{multline*}
  \| (P_X^0 - P_Y^0)(Z) \|_{\ell_2(\N^d)} 
  \le 
  \| (P_C - P_{\tilde C})C_Z \| + \| C_Z - \tilde C_Z \|  \\
  +\sqrt{2}\left( \frac{1}{\sigma_{r_1}^1} + \dots + \frac{1}{\sigma_{r_d}^d} \right) \|X-Y\|_{\ell_2(\N^d)}\| Z \|_{\ell_2(\N^d)}.  
\end{multline*}
It remains to estimate $\| C_Z - \tilde C_Z \|$ and $\| (P_C - P_{\tilde C})C_Z \|$. Using again a telescopic expansion of
\[
C_Z - \tilde C_Z = Z \times_1 (U^1)^\T \times_2 \dots \times_d (U^d)^\T - Z\times_1 (\tilde U^1)^\T \times_2 \dots \times_d (\tilde U^d)^T ,
\]
one obtains in a similar way as above that
\[
\| C_Z - \tilde C_Z \| \le \sqrt{2}\left( \frac{1}{\sigma_{r_1}^1} + \dots + \frac{1}{\sigma_{r_d}^d} \right)\|X-Y\|_{\ell_2(\N^d)} \| Z \|_{\ell_2(\N^d)}.
\]
We need to bound $\| (P_C - P_{\tilde C})C_Z \|$ in terms of $\|X-Y\|_{\ell_2(\N^d)}$. Note that
\[
X-Y =  C \times_1 U_1 \times_2 \dots \times_d U_d- \tilde C \times_1 \tilde U_1 \times_2 \dots \times_d \tilde U_d
\]
where $U_\mu$ and $\tilde U_\mu$ satisfy the assumptions in \Cref{lm: basis inequality}. It follows that $\|C-\tilde C\| \leq \sqrt 2 \|X-Y\|_{\ell_2(\N^d)}$. Hence by \eqref{eq:Mccurvest} and \Cref{prop:distequal} we have
\[
  \| (P_C - P_{\tilde C})C_Z \| 
  \leq 
  \frac{c}{\sigma} \|C-\tilde C\| \|C_Z\| \leq \frac{\sqrt 2 c}{\sigma} \|X-Y\|_{\ell_2(\N^d)} \|Z\|_{\ell_2(\N^d)}.
\]
In total, we have 
\[
  \| (P_X^0 - P_Y^0)(Z) \|_{\ell_2(\N^d)} 
  \le 
  \sqrt{2}\left(\frac{c}{\sigma}+ \frac{2}{\sigma_{r_1}^1} + \dots + \frac{2}{\sigma_{r_d}^d} \right) \|X-Y\|_{\ell_2(\N^d)}\| Z \|_{\ell_2(\N^d)}.  
\]
In summary, this allows to conclude the asserted curvature estimate.
\end{proof}

In Assumption~\ref{property:curvature} we also need an estimate for the projection $\id - P_X$.

\begin{proposition}\label{prop: curvatur estimate 2}
  Assume a curvature estimate of the form
\[
\| (\id - P_{\tilde C})(C-\tilde C) \| \le \frac{c}{\hat\sigma} \| C - \tilde C \|^2
\quad\text{for all $C\in \mathcal{M}_\mathrm{c}$ and $\tilde C\in\overline{\mathcal{M}_\mathrm{c}}$,}
\]
where $\hat\sigma = \dist(C,\overline{\mathcal{M}_\mathrm{c}} \setminus \mathcal{M}_\mathrm{c})$ and where $c>0$ is independent of $C, \tilde C$.
  Let $X,Y \in \mathcal M$ with corresponding tangent space projections $P_X$ and~$P_Y$. Then
  \[
  \| (\id - P_X)(X-Y) \|_{\ell_2(\N^d)}\le \sqrt{\frac{c^2}{\sigma^2}+\sum_{\mu = 1}^d \frac{1}{{(\sigma^\mu_{r_\mu})^2}}}  
  \| X - Y \|_{\ell_2(\N^d)}^2
  \le \frac{\sqrt{d+c^2}}{\sigma}\| X - Y \|_{\ell_2(\N^d)} ^2
  \]
  where $\sigma =  \dist(X,\overline{\mathcal M}^w \setminus \mathcal M)$.
\end{proposition}
\begin{proof}
  We use the same notation as in the proof of \Cref{prop: curvature estimate}.
  We decompose the identity into
  \begin{equation}\label{eq:sumdecomposition}
    \id = P_{\mathcal U^1\otimes \dots\otimes \mathcal U^d}
     + P_{(\mathcal U^1)^\perp\otimes \dots\otimes \mathcal U^d} 
     + P_{\mathcal \ell_2(\N)\otimes(\mathcal U^2)^\perp \dots\otimes \mathcal U^d} 
     + \ldots
     + P_{\mathcal \ell_2(\N^{d-1})\otimes(\mathcal U^d)^\perp}.
  \end{equation}
  Then 
  \begin{multline*}
    (\id - P_X)(X-Y)
    =(P_{\mathcal U^1\otimes \dots\otimes \mathcal U^d}-P^0_X)(X-Y)
    + (P_{ (\mathcal U^1)^\perp\otimes \dots\otimes \mathcal U^d} - P^1_X )(X-Y)\\
    \quad + (P_{\mathcal \ell_2(\N)\otimes(\mathcal U^2)^\perp \otimes \dots\otimes \mathcal U^d} - P^2_X )(X-Y)
    + \ldots
    + (P_{\mathcal \ell_2(\N^{d-1})\otimes(\mathcal U^d)^\perp}- P^d_X)(X-Y)
  \end{multline*}
  holds. For the first summand, we have
  \[
    (P_{\mathcal U^1\otimes \dots\otimes \mathcal U^d}-P^0_X)(X-Y)
    =
    \left((\id-P_C)(X-Y) \times_1 (U^1)^\T \times_2 \dots \times_d (U^d)^\T \right)\times_1 U^1 \times_2 \dots \times_d U^d
  \]
  where $\|(X-Y) \times_1 (U^1)^\T \times_2 \dots \times_d (U^d)^\T\|\leq \|X-Y\|_{\ell_2(\N^d)}$. Since $\dist(X,\overline{\mathcal M}^w \setminus \mathcal M) = \dist(C,\overline{\mathcal{M}_\mathrm{c}} \setminus \mathcal{M}_\mathrm{c})$ by \Cref{prop:distequal}, we have
  \[
    \|(P_{\mathcal U^1\otimes \dots\otimes \mathcal U^d}-P^0_X)(X-Y)\|_{\ell_2(\N^d)}
    \leq
    \frac{c}{\sigma}\|X-Y\|^2_{\ell_2(\N^d)},
  \]
  where we use that $Y \times_1 (U^1)^\T \times_2 \dots \times_d (U^d)^\T \in \overline{\mathcal{M}_\mathrm{c}}$.
  For the next summand, we obtain
  \[
  (P_{ (\mathcal U^1)^\perp\otimes \dots\otimes \mathcal U^d} - P^1_X )(X-Y)
  =
  P_{(\mathcal U^1)^\perp}\otimes (\id_{\ell_2(\N^{d-1})} - P_{\mathcal V^1} )(X- \id_{\ell_2(\N)}\otimes P_{ \mathcal U^2\otimes \dots\otimes \mathcal U^d} Y).
  \]
  Then, $\|X- \id_{\ell_2(\N)}\otimes P_{ \mathcal U^2\otimes \dots\otimes \mathcal U^d} Y\|_{\ell_2(\N^d)}\leq \|X- Y\|_{\ell_2(\N^d)}$ and $\id_{\ell_2(\N)}\otimes P_{ \mathcal U^2\otimes \dots\otimes \mathcal U^d} Y\in \overline{\mathcal M}^w$. If $\id_{\ell_2(\N)}\otimes P_{ \mathcal U^2\otimes \dots\otimes \mathcal U^d} Y\in \mathcal M$, then we have corresponding spaces $\tilde{\mathcal U}^1$ and $\tilde{\mathcal V}^1$ and  
  \[
    \|P_{\mathcal V^1}-P_{\tilde{\mathcal V}^1}\|_{\ell_2(\N^d)\to {\ell_2(\N^d)}}
    \leq\frac{1}{\sigma_{r_1}^1} \|X-\id_{\ell_2(\N)}\otimes P_{ \mathcal U^2\otimes \dots\otimes \mathcal U^d}Y\|_{\ell_2(\N^d)} 
    \leq\frac{1}{\sigma_{r_1}^1} \|X-Y\|_{\ell_2(\N^d)} .
  \]
  As a consequence,
  \[
    \|(P_{ (\mathcal U^1)^\perp\otimes \dots\otimes \mathcal U^d} - P^1_X )(X-Y)\|_{\ell_2(\N^d)}
    \leq \frac{1}{\sigma_{r_1}^1}\|X-Y\|_{\ell_2(\N^d)}^2.
  \]
  If $\id_{\ell_2(\N)}\otimes P_{ \mathcal U^2\otimes \dots\otimes \mathcal U^d} Y\notin \mathcal M$, the same estimate follows by continuity of the linear operator $P_{ (\mathcal U^1)^\perp\otimes \dots\otimes \mathcal U^d} - P^1_X $.
  Similar considerations show
  \[
    \|(P_{\ell_2(\N^{\mu-1})\otimes (\mathcal U^{\mu})^\perp\otimes \dots\otimes \mathcal U^d} - P^{\mu}_X )(X-Y)\|_{\ell_2(\N^d)}
    \leq \frac{1}{\sigma_{r_\mu}^\mu}\|X-Y\|_{\ell_2(\N^d)}^2 
  \]
  for $\mu=2,\ldots,d$.
  Finally, we obtain
  \[
    \| (\id - P_X)(X-Y) \|_{\ell_2(\N^d)}  \le \sqrt{\frac{c^2}{\sigma^2}+\sum_{\mu = 1}^d \frac{1}{{(\sigma^\mu_{r_\mu})^2}}}  \| X - Y \|^2_{\ell_2(\N^d)}
  \]
  utilizing that the images of the operators appearing in \eqref{eq:sumdecomposition} are orthogonal. The final inequality follows with \Cref{prop:distequal}.
\end{proof}

\subsection{Application to tensor train manifolds}

In order to give the above estimates a more concrete meaning, we now consider the popular example of the fixed-rank tensor-train~(TT) format discussed in the introduction. Here $\mathcal{M}_\mathrm{c}= \mathcal{M}_\bfk$ consists of all finite dimensional tensors $C\in \R^{r_1\times\dots\times r_d}$ with the fixed TT rank $\bfk$ of the form~\eqref{eq: TT-format}. 
We denote the resulting manifold $\mathcal M$ in \eqref{eq: Tucker format} by $\mathcal M_{\bfr,\bfk}$. It thus contains infinite tensors in $\ell_2(\N^d)$ of ``outer'' multilinear rank $\bfr = (r_1,\dots,r_d)$ and ``inner'' TT rank $\bfk = (k_1,\dots, k_{d-1})$.
This can be seen as a special case of the hierarchical tensor format with linear dimension tree, see \cite[Rem.~2.27]{Bachmayr:23}.

\begin{proposition}{\label{prop: SubspaceTTCurvature}}
  Let $X,Y \in \mathcal M_{\bfr,\bfk}$ with corresponding tangent space projections $P_X$ and~$P_Y$. Then
  \[
    \begin{aligned}
  \| P_X - P_Y \|_{\ell_2(\N^d) \to \ell_2(\N^d)} &\le \frac{2d(3\sqrt{2}+1)}{\sigma} \| X - Y \|_{\ell_2(\N^d)},  \\
  \| (\id - P_X)(X-Y) \|_{\ell_2(\N^d)}  &\le \frac{\sqrt{2d- 1}}{\sigma}\| X - Y \|^2_{\ell_2(\N^d)} ,
    \end{aligned}
  \]
  where $\sigma =  \dist(X,\overline{\mathcal M_{\bfr,\bfk}}^w \setminus \mathcal M_{\bfr,\bfk})$.
\end{proposition}
The result follows directly follows from Propositions~\ref{prop: curvature estimate} and~\ref{prop: curvatur estimate 2} and the following refined curvature estimates for the finite-dimensional TT manifold $\mathcal M_\bfk$, which under the given assumptions seem to be new.
\begin{proposition}\label{prop: TT curvature}
  Let $\mathcal M_{\bfk}\subset{\R^{N_1\times\dots\times N_d}}$ be a finite-dimensional TT manifold of fixed TT rank $\bfk$. 
  Let $X,Y\in {\mathcal M}_{\bfk}$ and $\sigma = \dist(X,\overline{{\mathcal M}}_\bfk\setminus{\mathcal M}_{\bfk})$. Then
  \[
 \max_{\| Z\| = 1} \|(P_X-P_Y) Z\| \leq \frac{4d}{\sigma}\|X-Y\|,\quad 
  \|(\id-P_X)(X-Y)\| \leq \frac{\sqrt{d-1}}{\sigma}\|X-Y\|^2
  \]
and
  \[
  \|(\id-P_X)(X-Y)\|  \leq \frac{\sqrt{d-1}}{\sigma}\|X-Y\|^2.
  \]
Furthermore, the last inequality holds more generally for $Y\in \overline{\mathcal M}_\bfk$.
\end{proposition}

The proof is given in the appendix.

\section{Application to the model problem}\label{sec:applmodelproblem}
We now return to the model problem \eqref{eq:modelproblem2DLRA} under the regularity assumption~\ref{property:regularity}.
\begin{problem}\label{modelproblem}
  Given $f\in L_2(0,T;L_2(\Omega))$ and
  $u_0 \in \mathcal M \cap H_0^1(\Omega)$, find
  \[
   u \in W(0,T;H_0^1(\Omega),L_2(\Omega)) = \{ u \in L_2(0,T; H_0^1(\Omega)) \colon  u' \in L_2(0,T; \mathcal V^*) \}
  \]
  such that for almost all $t \in [0,T]$,
  \begin{equation}\label{eq:modelprob1}
    \begin{aligned}
    u(t)  &\in \mathcal M,\\
      \langle u'(t), v\rangle + a(u(t),v;t) &= \langle f(t), v \rangle \quad \text{for all $v \in T_{u(t)} \mathcal M \cap H_0^1(\Omega)$}, \\
      u(0) &= u_0.
    \end{aligned}
  \end{equation}
  \end{problem}

Here 
\[
a(u,v;t) =  \int_\Omega (B(t)\nabla u(x))\cdot \nabla v(x) \,dx  
\]
with a symmetric positive definite matrix $B(t)$ that is entrywise Lipschitz continuous in~$t$
and
$\mathcal M$ is a manifold of functions in $L_2(\Omega) = L_2(\Omega_1 \times \dots \times \Omega_d)$ as described in \eqref{eq: Tucker format}, that is, $u\in\mathcal M$ is of the form
\begin{equation}\label{eq:tuckerfunction}
  u = \sum_{k_1 = 1}^{r_1} \cdots \sum_{k_d = 1}^{r_d} C(k_1,\dots,k_d) \, u^1_{k_1} \otimes \dots  \otimes  u^d_{k_d},
\end{equation}
with $C\in\mathcal M_\mathrm c$.
\subsection{Discussion of main assumptions}\label{sec:mainassumptions}
Our goal is to apply Theorems~\ref{thm:existandunique},~\ref{thm:stability} and~\ref{thm:convergence_space-discr} to \Cref{modelproblem}.  It suffices to verify the Assumptions~\ref{property:cone}--\ref{property:splitting}, \ref{property:discr_approximation}, and~\ref{property:discr_compat}.
Assumption~\ref{property:cone} holds since $\mathcal{M}_\mathrm{c}$ is a cone, which follows from the invariance assumption~\eqref{eq: invariance Mhat}.
We already proved \ref{property:curvature} in \Cref{prop: curvature estimate} and \Cref{prop: curvatur estimate 2} assuming that corresponding curvature bounds for $\mathcal{M}_\mathrm{c}$ are available.
Indeed, in the special case where $\mathcal{M}_\mathrm{c}$ is the manifold of tensors with constant TT-rank~$\mathcal M_\mathbf{k}$, such curvature bounds are stated in \Cref{prop: SubspaceTTCurvature}. 
We now consider the remaining Assumptions~\ref{property:compat} and~\ref{property:splitting} as well as~\ref{property:discr_approximation} and~\ref{property:discr_compat}. We make use of the following well-known technique for estimating norms of factors in low-rank representations, see for example \cite{Uschmajew11}.

\begin{lemma}\label{lm: singularvector H1norm}
  Let $u\in H^1_0(\Omega)$ admit a singular value decomposition
  \[
    u(x) = \sum_{k =1}^{r_\mu} \sigma_k u_{1,k}(x_\mu) u_{2,k}(x_{\{1,\ldots, d\}\setminus\{\mu\}})
  \] 
  with respect to the $\mu$-th variable. Then the singular vectors satisfy $u_{1,k}\in H_0^1(\Omega_\mu)$ and $ u_{2,k} \in H_0^1(\bigtimes_{\nu\neq \mu} \Omega_\nu)$ with
  \[
    \|u_{1,k}\|_{H_0^1(\Omega_\mu)}\leq \frac{1}{\sigma_k} \|u\|_{H_0^1(\Omega)}\quad \text{and} \quad
    \|u_{2,k}\|_{H_0^1(\bigtimes_{\nu\neq \mu} \Omega_\nu)}\leq \frac{1}{\sigma_k} \|u\|_{H_0^1(\Omega)}.
  \]
\end{lemma}
\begin{proof}
  We state the proof for $\mu = 1$. Then
  \[\sigma_k u_{1,k}(x_1) = \int_{\bigtimes_{\nu = 2}^d \Omega_\nu} u(x_1,x_2,\ldots,x_d) u_{2,k}(x_2,\ldots, x_d)\, d(x_2,\ldots, x_d)
  \]
  and
  \[\sigma_k u_{2,k}(x_2,\ldots, x_d) = \int_{\Omega_1} u(x_1,x_2,\ldots,x_d) u_{1,k}(x_1)\, dx_1.
  \]
  By the Cauchy-Schwarz inequality, we have
\begin{multline*}
  \sigma_k^2\|u_{1,k}\|_{H_0^1(\Omega_1)}^2
  =
  \sigma_k^2\int_{\Omega_1} \abs{\nabla_{x_1} u_{1,k} }^2 \,dx_1 
  =
   \int_{\Omega_1} \abs{\int_{\bigtimes_{\nu = 2}^d \Omega_\nu}\nabla_{x_1}
   u\; u_{2,k}\, d(x_2, \ldots,x_d)
   }^2 \,dx_1\\
   \leq
   \|u_{2,k}\|_{L_2((0,1)^{d-1})}^2\int_\Omega \abs{\nabla_{x_1}u}^2\,d(x_1,\ldots, x_d) 
   = \|u\|^2_{H_0^1(\Omega_1)\otimes L^2(\bigtimes_{\nu = 2}^d \Omega_\nu)}
   \leq \|u\|^2_{H_0^1(\Omega)}.
\end{multline*}
This proves the first estimate. The other one follows in analogy.
\end{proof}

\subsubsection{Assumption \ref{property:compat}: Compatibility of tangent spaces}

We now verify~\ref{property:compat1}.

\begin{lemma}
  Let $u\in H_0^1(\Omega)\cap \mathcal M$ and $v \in H_0^1(\Omega)\cap T_u\mathcal M$. Then there exists an admissible curve $\varphi(t)\in H_0^1(\Omega)\cap \mathcal M$ for $\abs{t} $ small enough with $\varphi(0)=u$ and $\varphi'(0)=v$.
\end{lemma}

\begin{proof}
  Let $u$ be of the form~\eqref{eq:tuckerfunction}. Then by \Cref{lm: singularvector H1norm} all basis functions $u_{k_\mu}^\mu\in H_0^1(\Omega_\mu)$.
  We write $v$ is in~\eqref{eq: tangent vectors}, that is, $v = v_0+v_1+\ldots+v_d$, where 
  \[
    v_0 = \sum_{k_1 = 1}^{r_1} \cdots \sum_{k_d = 1}^{r_d} \dot{C}(k_1,\dots,k_d) \, u^1_{k_1} \otimes \dots  \otimes  u^d_{k_d}
  \]
  with $\dot C\in T_C\mathcal M_\mathrm c$ and 
  \[
    v_\mu = \sum_{k_1 = 1}^{r_1} \cdots \sum_{k_d = 1}^{r_d} C(k_1,\dots,k_d) \, u^1_{k_1} \otimes \dots \otimes u^{\mu-1}_{k_{\mu-1}}\otimes \dot u^{\mu}_{k_\mu} \otimes u^{\mu+1}_{k_{\mu+1}}  \otimes \dots \otimes  u^d_{k_d}
  \]
  where $\dot u_{k_\mu}^\mu$ is orthogonal to all $u_1^\mu,\ldots, u_{r_\mu}^\mu$. 
  By similar reasoning as in \Cref{lm: singularvector H1norm} one can show that $\dot u_{k_\mu}^\mu\in H_0^1(\Omega_\mu)$.
  Furthermore, there exists a curve $D(t)\in \mathcal M_\mathrm c$ for $\abs{t}$ small enough such that $D(0)=C$ and $D'(0) = \dot C$. We now choose 
  \[
    \varphi(t) = \sum_{k_1 = 1}^{r_1} \cdots \sum_{k_d = 1}^{r_d} D(k_1,\dots,k_d)(t) \, (u^1_{k_1}+t \dot u^1_{k_1}) \otimes \dots \otimes  (u^d_{k_d}+t \dot u^d_{k_d}),
  \]
  which is a differentiable curve in $\mathcal M\cap H_0^1(\Omega)$ for small enough $\abs t$. By the product rule, it satisfies $\varphi(0)=u$ and $\varphi'(0)=v$.
\end{proof}

Assumption~\ref{property:compat2}, which states that for $u\in H_0^1(\Omega) \cap \mathcal M$ the $L_2$-orthogonal projection onto its tangent space is also a bounded operator with respect to the $H_0^1$-norm, follows with a similar technique as \Cref{lm: singularvector H1norm}.
\begin{proposition}\label{prop:H1projectorcompat}
  Let $\mathcal M$ be of the form~\eqref{eq: Tucker format}, let $u\in H^1_0(\Omega)\cap \mathcal M$ and $v\in H^1_0(\Omega)$. Let $P_u^0, \ldots, P_u^d$ be the projections in~\eqref{eq: projector splitting}. Then
  \[
    \|P_u^\mu v\|_{H^1_0(\Omega)}\leq c    \|v\|_{H^1_0(\Omega)}
  \]
\end{proposition}

\begin{proof}
  First, we consider $P_u^0$. We note the norm bound 
  \[
    \|C_Z\| \leq \|Z\|_{\ell_2(\N^d)}
  \]
in the definition of $P_X^0$ in \Cref{prop: tangent space projection}. In terms of the represented function, 
\[
  \|C_v\| \leq \|v\|_{L_2(\Omega)}.
\]
Furthermore, we have
\[
  \|P_C(C_v)\| \leq\|C_v\|
\]
since $P_C$ is an $\ell_2$-orthogonal projection. We now consider the summands of
\[
\|P_u^0 v\|^2_{H_0^1(\Omega)} =   
\|P_u^0 v\|^2_{H_0^1{(\Omega_1)\otimes L_2(\bigtimes_{\mu = 2}^d \Omega_\mu)}}  +
\ldots +
\|P_u^0 v\|^2_{L_2(\bigtimes_{\mu = 1}^{d-1} \Omega_\mu)\otimes H_0^1{(\Omega_d)}} 
\]
independently.
Let 
\begin{equation}\label{eq: H1L2orthodec}
  u(x) = \sum_{k =1}^{r_1} \sum_{\ell = 1}^{r_1} a_{k\ell} u_{1,k}(x_1) u_{2,\ell}(x_2,\ldots, x_d)
\end{equation}
be a decomposition of $u$ separating the first variable, where $\{u_{1,k}\}$ and $\{u_{2,\ell}\}$ are $L_2$-orthonormal (as, for example, in a singular value decomposition).
After an orthogonal change of basis, we may assume the vectors $u_{1,k}$ to be both $L_2$-orthonormal and $H^1_0$-orthogonal. Note that the basis vectors are given by
\[
  u_{1,k} = \int_{\bigtimes_{\mu = 2}^d \Omega_\mu} u\; \sum_{\ell= 1}^{r_1} b_{k\ell} u_{2,\ell}\, d(x_2, \ldots, x_d)
  \quad\text{and}\quad
  u_{2,\ell} = \int_{\Omega_1} u\; \sum_{k= 1}^{r_1} b_{k\ell} u_{1,k}\, dx_1
\]
where $\sum_{\ell = 1}^{r_1} a_{k_1\ell} b_{k_2\ell} = \delta_{k_1,k_2}$. 

The singular values of the corresponding decomposition satisfy $\sigma_{k}\geq \sigma$ for $k = 1,\ldots, r_\nu$, and hence the vectors in~\eqref{eq: H1L2orthodec} satisfy 
\begin{equation}\label{eq: H1normvec}
  \begin{aligned}
  \sigma\|u_{1,k}\|_{H_0^1(\Omega_1)}
  &\leq 
  \|u\|_{H_0^1(\Omega_1)\otimes L_2(\bigtimes_{\mu = 2}^d \Omega_\mu)}  , \\
  \sigma\|u_{2,k}\|_{H_0^1(\bigtimes_{\mu = 2}^d \Omega_\mu)}
 & \leq 
  \|u\|_{L_2(\Omega_1)\otimes H_0^1(\bigtimes_{\mu = 2}^d \Omega_\mu)} ,
  \end{aligned} 
\end{equation}
since $\sigma \norm{\sum_{k= 1}^{r_1} b_{k\ell} u_{1,k}}_{L_2{(\bigtimes_{\mu = 2}^d \Omega_\mu)}} \leq 1$, $\sigma \norm{\sum_{\ell= 1}^{r_1} b_{k\ell} u_{2,\ell}}_{L_2(\Omega_1)}\leq 1$ and 
by the last argument of the proof of \Cref{lm: singularvector H1norm}.
By $L_2$- and $H_0^1$-orthogonality and~\eqref{eq: H1normvec}, we obtain the estimate
\begin{align*}
  \|P_u^0 &v\|^2_{H_0^1{(\Omega_1)\otimes L_2(\bigtimes_{\mu = 2}^d \Omega_\mu)}} \\
  &=
  \biggl\|\sum_{k_1 = 1}^{r_1} \cdots \sum_{k_d = 1}^{r_d} P_C(C_v)(k_1,\dots,k_d) \, \partial _{x_1}u^1_{k_1} \otimes u^2_{k_2} \otimes \dots \otimes u^d_{k_d}\biggr\|^2_{L_2(\Omega)}   \\
  &= 
  \sum_{k_1 = 1}^{r_1} \biggl\|\sum_{k_2 = 1}^{r_2} \cdots \sum_{k_d = 1}^{r_d} P_C(C_v)(k_1,\dots,k_d) \, u^2_{k_2} \otimes \dots \otimes u^d_{k_d}\biggr\|^2_{L_2(\bigtimes_{\mu = 2}^d \Omega_\mu)} \|u^1_{k_1}\|^2_{H_0^1(\Omega_1)} \\
  &\leq
  \frac{1}{\sigma^2}
  \|P_C(C_v)\|^2
  \|u\|^2_{H_0^1(\Omega_1)\otimes L_2(\bigtimes_{\mu = 2}^d \Omega_\mu)}.
\end{align*}
Using the Poincar\'e inequality, we have
\[
  \|P_C(C_v)\| \leq \norm{ C_v} = \norm{v}_{L_2(\Omega)} \leq  c_\Omega \norm{ v}_{H^1_0(\Omega)}
\]
with a $c_\Omega>0$ depending only on $\Omega$.
We thus arrive at
\[
  \|P_u^0 v\|_{H_0^1{(\Omega_1)\otimes L_2(\bigtimes_{\mu = 2}^d \Omega_\mu)}} \leq \frac{c_\Omega}{\sigma}
  \|v\|_{H_0^1(\Omega)}
  \|u\|_{H_0^1(\Omega_1)\otimes L_2(\bigtimes_{\mu = 2}^d \Omega_\mu)}.
\]
Similarly,
\begin{multline*}
  \|P_u^0 v\|_{L_2(\bigtimes_{\nu = 1}^{\mu-1} \Omega_\nu)\otimes H_0^1{(\Omega_\mu)\otimes L_2(\bigtimes_{\nu = \mu+1}^d \Omega_\nu)}} \\ \leq \frac{c_\Omega}{\sigma}
  \|v\|_{H_0^1{(\Omega)}}
  \|u\|_{L_2(\bigtimes_{\nu = 1}^{\mu-1} \Omega_\nu)\otimes H_0^1{(\Omega_\mu)\otimes L_2(\bigtimes_{\nu = \mu+1}^d \Omega_\nu)}}.
\end{multline*}
This yields
\[
  \|P_u^0 v\|_{H_0^1(\Omega)} \leq \frac{c_\Omega}{\sigma}
  \|v\|_{H_0^1{(\Omega)}}
  \|u\|_{H_0^1(\Omega)}.
\]

Next we consider $P_u^1$; the estimates for the projections $P_u^2,\ldots, P_u^d$ follow in analogy.
The action of $P_u^1$ is given by the tensor product of $L_2$-orthogonal projections
\[
  P_u^1(v)  =  (\id-P_1)\otimes P_2 v ,
\]
where  
\[
  (P_1 w )(x_1)
  =
\sum_{k=1}^{r_1}u_{1,k}(x_1) \int_{\Omega_1} u_{1,k}(y_1)w(y_1) \, dy_1
\]
and 
\[
(P_2 w )(x_2,\ldots, x_d)
=
\sum_{k=1}^{r_1}u_{2,k}(x_2,\ldots, x_d) \int_{\bigtimes_{\mu = 2}^d \Omega_\mu}\!\!\!\!\!\! u_{2,k}(y_2,\ldots, y_d)w(y_2,\ldots, y_d) \, d(y_2, \ldots ,y_d).
\]
We again use the decomposition~\eqref{eq: H1L2orthodec} with $L_2$-orthonormal and $H_0^1$-orthogonal sets of vectors $\{u_{1,k}\colon k=1,\ldots, r_1\}$  and $\{u_{2,k}\colon k=1,\ldots, r_1\}$. Using orthogonality, we get the estimate 
\begin{align*}
  \|P_1 w\|^2_{H_0^1(\Omega_1)}  
&=
\sum_{k=1}^{r_1}\|u_{1,k}\|^2_{H_0^1(\Omega_1)}  \left(\int_0^1 u_{1,k}(y_1)w(y_1) \, dy_1\right)^2
\\
&\leq
\frac{1}{\sigma^2}\|u\|^2_{H_0^1(\Omega_1)\otimes L_2(\bigtimes_{\mu = 2}^d \Omega_\mu)}\|w\|^2_{L_2(\Omega_1)} 
\end{align*}
and similarly
\[
  \|P_2 w\|_{H_0^1(\bigtimes_{\mu = 2}^d \Omega_\mu)}  \leq \frac{1}{\sigma}\|u\|_{{L_2(\Omega_1)\otimes H_0^1(\bigtimes_{\mu = 2}^d \Omega_\mu)}} \|w\|_{L_2(\bigtimes_{\mu = 2}^d \Omega_\mu)}.
\]
We finally obtain
\begin{align*}
  \|P_u^1v\|^2_{H^1_0(\Omega)} 
  &=
  \|P_u^1v\|^2_{H_0^1(\Omega_1)\otimes L_2(\bigtimes_{\mu = 2}^d \Omega_\mu)}
  +
  \|P_u^1v\|^2_{L_2(\Omega_1)\otimes H_0^1(\bigtimes_{\mu = 2}^d \Omega_\mu)} \\
  &=
  \|((\id-P_1)\otimes P_2) v\|^2_{H_0^1(\Omega_1)\otimes L_2(\bigtimes_{\mu = 2}^d \Omega_\mu)} \\
  &\qquad +
  \|((\id-P_1)\otimes P_2) v \|^2_{L_2(\Omega_1)\otimes H_0^1(\bigtimes_{\mu = 2}^d \Omega_\mu)} \\
  &\leq
  \left(\frac{1}{\sigma}\|u\|_{H_0^1(\Omega_1)\otimes L_2(\bigtimes_{\mu = 2}^d \Omega_\mu)}\|v\|_{L_2(\Omega)} + \|v\|_{{H_0^1(\Omega_1)\otimes L_2(\bigtimes_{\mu = 2}^d \Omega_\mu)}}\right)^2
  \\ &\qquad +
  \frac{1}{\sigma^2}\|u\|^2_{L_2(\Omega_1)\otimes H^1_0(\bigtimes_{\mu = 2}^d \Omega_\mu)}\|v\|^2_{L_2(\Omega)}
  \\
  &\leq 
  2\left(1+\frac{c_{\Omega_1}}{\sigma^2}\|u\|^2_{H^1_0(\Omega)}\right)\|v\|^2_{H_0^1(\Omega_1)\otimes L_2(\bigtimes_{\mu = 2}^d \Omega_\mu)},
\end{align*}
where we have again used Poincar\'e's inequality.
\end{proof}

We have thus verified
Assumption~\ref{property:compat}.

\subsubsection{Assumption \ref{property:splitting}: Operator decomposition}
The bilinear form $a(\cdot, \cdot;t)$ can be written as
$a(\cdot, \cdot;t) = a_1(\cdot, \cdot;t) + a_2(\cdot, \cdot;t)$ with 
\[
a_1(u,v;t) = \int_\Omega \sum_{\mu = 1}^d \bigl(b_{\mu\mu}(t)\nabla_{x_\mu}u(x)\bigr) \cdot \nabla_{x_\mu} v(x) \; dx
\] 
and
\[
a_2(u,v;t) = \int_\Omega \sum_{\substack{\mu, \nu = 1\\ \mu \neq \nu}}^d \bigl(b_{\mu\nu}(t)\nabla_{x_\mu}u(x)\bigr)\cdot \nabla_{x_\nu} v(x) \; dx,
\] 
where $b_{\mu\nu}(t)$ are the corresponding blocks of the matrix $B(t)$.
These bilinear forms in turn define linear operators $A_1$ and $A_2$.
For \ref{property:partA1}, we can use a similar technique as in~\cite{Bachmayr21}, albeit on the more complicated manifold $\mathcal M$. First, we show that the strong version of $A_1$, when defined, maps to the tangent space.
\begin{lemma}\label{lm:tangentspacepropertystrong}
  Let $\mathcal M$ be of the form~\eqref{eq: Tucker format}  and  $u\in H^2(\Omega) \cap H^1_0(\Omega)\cap \mathcal M$. Then for matrices~$a_\mu$ of the corresponding size, we have $\sum_{\mu = 1}^d \nabla_{x_\mu} \cdot (a_\mu\nabla_{x_\mu} u) \in T_u{\mathcal M}$.
\end{lemma}
\begin{proof}
  Since $u\in \mathcal M$, we can write $u$ as in~\eqref{eq:tuckerfunction}
  and
  \[
    \nabla_{x_1} \cdot (a_1\nabla_{x_1} u) 
    =  
    \sum_{k_1 = 1}^{r_1} \cdots \sum_{k_d = 1}^{r_d} C(k_1,\dots,k_d) \nabla_{x_1} \cdot (a_1\nabla_{x_1}u^1_{k_1})\otimes u^2_{k_2}\otimes \dots \otimes u^d_{k_d}. 
  \]
  We define $\varphi(t) = u + t \nabla_{x_1} \cdot (a_1\nabla_{x_1} u)$. For sufficiently small $\abs{t}$, the Gramian of the system $\{u^1_1+t\nabla_{x_1} \cdot (a_1\nabla_{x_1}u_1^1),u^1_2+t\nabla_{x_1} \cdot (a_1\nabla_{x_1}u_2^1),\dots,u^1_{r_1}+t\nabla_{x_1} \cdot (a_1\nabla_{x_1}u_{r_1}^1) \}$ is invertible and hence $\varphi(t)\in\mathcal M$ with $\varphi'(t) = \nabla_{x_1} \cdot (a_1\nabla_{x_1} u) $. Thus $\nabla_{x_1} \cdot (a_1\nabla_{x_1} u) \in T_u\mathcal M$. Analogously, $\nabla_{x_\mu} \cdot (a_\mu\nabla_{x_\mu} u) \in T_u\mathcal M$ for $\mu = 1,\ldots, d$ and by linearity, $\sum_{\mu = 1}^d \nabla_{x_\mu} \cdot (a_\mu\nabla_{x_\mu} u) \in T_u{\mathcal M}$.
\end{proof}

The assumption \ref{property:partA1} now follows by a density argument.
For a proof, choose a sequence $(u_n) \subset \mathcal M \cap H^2(\Omega) \cap H^1_0 (\Omega)$ converging to $u$ in
$H^1_0 (\Omega)$-norm. 
Then for $v \in H^1_0 (\Omega)$, we have
\[
a_1(u_n, v;t) = \langle A_1(t)u_n, v\rangle = \langle A_1(t)u_n, P_{u_n} v \rangle = a_1(u_n, P_{u_n} v;t)
\]
since $A_1(t)u_n \in T_{u_n}\mathcal M$
 by \Cref{lm:tangentspacepropertystrong}. Moreover,
\[
a_1(u_n, P_{u_n} v;t) = a_1(u, P_uv) + a_1(u,(P_{u_n} - P_u)v) + a_1(u_n - u, P_{u_n} v).
\]
We have $P_{u_n} v\to P_u v $ strongly in $L_2(\Omega)$ by \Cref{prop: curvature estimate}, and \Cref{prop:H1projectorcompat} yields $\limsup_n	\|P_{u_n} v\|_{	H^1_0}
< \infty$. Since $L_2(\Omega)$ is dense in $H^{-1}(\Omega)$ it follows that $P_{u_n} v \to P_u v $ weakly in
$H^1_0(\Omega)$ by a standard argument; see for example~\cite[Prop. 21.23(g)]{Zeidler90a}. Consequently,
$a_1(u_n, P_{u_n} v;t) \to a_1(u, P_uv; t)$. At the same time, $a_1(u_n, v;t) \to a_1(u, v;t)$, so we
have verified Assumption~\ref{property:partA1}.

For Assumption~\ref{property:partA2}, it suffices to verify mixed smoothness for~$u\in H_0^1(\Omega)\cap \mathcal M$.

\begin{lemma}
  Let $\mathcal M$ be of the form~\eqref{eq: Tucker format}, $u\in H^1_0(\Omega)\cap \mathcal M$, and $\sigma = \dist_{L_2}(u,\overline{\mathcal M}^w \setminus \mathcal M )$.
Then for $\nu\neq \mu$, we have $\|\nabla_{x_\mu}\nabla_{x_\nu} u\|_{L_2(\Omega)} \leq \frac{1}{2\sigma } \|u\|^2_{H^1_0(\Omega)}$.

\end{lemma}
\begin{proof}
  Let $u(x) = \sum_{k =1}^{r_\mu} \sigma_k u_{1,k}(x_\mu) u_{2,k}(x_{\{1,\ldots, d\}\setminus\{\mu\}})$ be a singular value decomposition of $u$ separating the $\mu$-th variable. Then $\sigma_{k}\geq \sigma$ for $k = 1,\ldots, r_\nu$. We will use the abbreviation $\Omega_{\mu^c} = \bigtimes_{\lambda\neq\mu}\Omega_\lambda$. On the one hand, by the triangle inequality, Young's inequality, and $\sigma \le \sigma_k$, we have
  \begin{align*}
    \|\nabla_{x_\mu}\nabla_{x_\nu} u\|_{L_2(\Omega)}
    &\leq 
    \sum_{k =1}^{r_\mu} \sigma_k
    \|\nabla_{x_\mu}\nabla_{x_\nu} u_{1,k} \otimes u_{2,k}\|_{L_2(\Omega)}
    \\
    &=
    \sum_{k =1}^{r_\mu} \sigma_k
    \|\nabla_{x_\mu} u_{1,k}\|_{L_2(\Omega_\mu)} \|\nabla_{x_\nu} u_{2,k}\|_{L_2(\Omega_{\mu^c})} \\
    &\le
    \sum_{k=1}^{r_\mu}\frac{\sigma_k ^2}{2\sigma}\left(
      \|\nabla_{x_\mu} u_{1,k}\|_{L_2(\Omega_\mu)}^2 + \|\nabla_{x_\nu} u_{2,k}\|^2_{L_2(\Omega_{\mu^c})}\right).
  \end{align*}
  On the other hand, by $L_2$-orthogonality of the singular vectors, we have
  \[
    \|u\|^2_{H_0^1(\Omega)} 
    =
    \int_{\Omega} \sum_{\lambda = 1}^d \abs{ \nabla_{x_\lambda} u(x)}^2\, dx
    =
    \sum_{k =1}^{r_\mu} \sigma_k ^2\left(
    \|\nabla_{x_\mu} u_{1,k}\|_{L_2(\Omega_\mu)}^2 +\sum_{\lambda\neq \mu} \|\nabla_{x_\lambda} u_{2,k}\|^2_{L_2(\Omega_{\mu^c})}\right),
  \]
  and hence $\|\nabla_{x_\mu}\nabla_{x_\nu} u\|_{L_2(\Omega)} \leq \frac{1}{2\sigma } \|u\|^2_{H^1_0(\Omega)}$ as asserted.
\end{proof}

 Assumption~\ref{property:partA2} now follows directly by integration by parts.

\subsubsection{Assumptions \ref{property:discr_approximation} and \ref{property:discr_compat}: Spatial discretizations}\label{sec:assumptionsBmodel}

  We now exhibit space discretizations that can be used to achieve convergence to the infinite dimensional solution as in \Cref{thm:convergence_space-discr}. It turns out to be sufficient for the discretization to allow the format~\eqref{eq: Tucker format}. This is a natural requirement, since otherwise one does not have the required product structure for using the low-rank approximation in practice.

  For short, let us denote $H_0^1(\Omega)=\mathcal V=V^1\otimes H^2\otimes \dots\otimes H^d\cap \dots \cap H^1\otimes  \dots\otimes H^{d-1} \otimes V^d$, where $V^\mu = H_0^1(\Omega_\mu)$ and $H^\mu = L_2(\Omega_\mu)$ for $\mu = 1,\ldots, d$. Then $V^1\otimes \dots\otimes V^d$ is a continuously and densely embedded subspace of $\mathcal V$. As the finite-dimensional subspaces, we choose 
  \begin{equation}\label{eq:discretemodelspaces}
      \mathcal V_h= V_h^1\otimes \dots\otimes V_h^d  \quad\text{satisfying}\quad
      \|P_{V^\mu_h}v^\mu-v^\mu\|_{V^\mu}\to 0  \quad \text{as} \quad h\to 0 \quad \text{for all $v^\mu\in V^\mu$}.    
  \end{equation}
   This can be a finite element space, but also any other discretization suitable for $\Omega_\mu$. As a consequence, for $v\in V^1\otimes \dots\otimes V^d$, we have
  \[
    \|v - P_{V^1_h\otimes \dots\otimes V^d_h}v\|_\mathcal V \leq C 
    \|v - P_{V^1_h\otimes \dots\otimes V^d_h}v\|_{V^1\otimes \dots\otimes V^d}
    \to 0 \quad \text{for }\quad h\to 0.
  \]
  Since $V^1\otimes \dots\otimes V^d$ is a dense subspace of $\mathcal V$, the $\mathcal V$-orthogonal projection onto $V^1_h\otimes \dots\otimes V^d_h$ satisfies~{\ref{property:discr_approximation}}(a). 
  For assumption~{\ref{property:discr_approximation}}(b), let $u\in \mathcal V\cap \mathcal M$. Then $u\in V^1\otimes \dots\otimes V^d$ by \Cref{lm: singularvector H1norm}. By~$\eqref{eq: invariance Mhat}$, we have that $P_{V^1_h\otimes \dots\otimes V^d_h}u\in \overline {\mathcal M}^w$. Hence there exists $u_h\in V^1_h\otimes \dots\otimes V^d_h \cap \mathcal M$ such that $\|u_h-P_{V^1_h\otimes \dots\otimes V^d_h}u\|_\mathcal V \leq \epsilon$ for any $\epsilon>0$. Thus $\|u_h-u\|_\mathcal V\to 0$ as $h\to 0$. Possibly after rescaling, we can thus construct a sequence $(u_h)$ that converges to $u$ in $\mathcal V$ as $h\searrow 0$ with $\|u_h\|_\mathcal V \leq \|u\|_\mathcal V$.

  Assumption {\ref{property:discr_compat}} follows immediately by noting that $\mathcal M\cap \mathcal V_h$ is of the same form as $\mathcal M$ and \Cref{thm: manifold} can be applied.
  
  \begin{remark}
  In practical numerical realizations, the elements of the discretization subspaces $\mathcal V_h$ in \eqref{eq:discretemodelspaces} 
  need to be represented in terms of suitable basis functions, usually obtained as tensor products of bases of each $V^\nu_h$. 
  While the results given here refer to the represented functions, the properties of the problem in terms of basis coefficients depend also
  on the condition number of the chosen basis. For tensor product bases, this condition number is the product of the condition number of the univariate bases, and thus in general depends exponentially on $d$ unless orthonormal bases are used for each $V^\nu_h$. A possible remedy for large $d$ is the use of nonstandard basis functions (such as wavelets) for finite element spaces. We refer to \cite{Bachmayr:23} for further details.
  \end{remark}

\subsection{Main results}

The main results for the model problem are the following specific versions of \Cref{thm:existandunique}, \Cref{thm:convergence_space-discr}, and \Cref{thm:stability}. They follow directly by applying the results of \Cref{sec:tensor train model} and \Cref{sec:mainassumptions}.

\begin{theorem}[Existence and uniqueness of solutions]\label{thm:mainmodel}
  Let $u_0$ have positive $L_2$-distance from $\overline{\mathcal M}^w \setminus \mathcal M$.
   There exist $T^* \in (0,T]$ and
    $u \in W(0,T^*;H_0^1(\Omega),L_2(\Omega)) \cap L_\infty(0,T^*;H_0^1(\Omega))$ 
  such that $u$ solves \Cref{modelproblem} on the time interval $[0,T^*]$, and its continuous representative $u\in C(0,T^*;{L_2(\Omega)})$ satisfies $u(t) \in \mathcal M$ for all $t \in [0,T^*)$. Here $T^*$ is maximal for the evolution on $\mathcal M$ in the sense that if $T^* < T$, then 
  \[
    \liminf_{t \to T^*} \; \inf_{v\in\overline{\mathcal M}^w \setminus \mathcal M } \|u(t)-v\|_{L_2(\Omega)} = 0.
  \]
  In either case, $u$ is the unique solution of \Cref{modelproblem} in $W(0,T^*;H_0^1(\Omega),{L_2(\Omega)})$.
  
  In particular, with $\sigma = \dist_{{L_2(\Omega)}}(u_0,\overline{\mathcal M}^w \setminus \mathcal M )$, there exists a constant $c>0$ such that $T^* \ge \min(\sigma^2 / c,T)$.
   
  The solution satisfies the following estimates:
  \begin{align*}
      \|u\|^2_{L_2(0,T^*;H_0^1(\Omega))}&\leq \|u_0\|_{L_2(\Omega)}^2+C_1\|f\|^2_{L_2(0,T^*;{L_2(\Omega)})},
      \\
      \|u'\|^2_{L_2(0,T^*;{L_2(\Omega)})}&\leq C_2\left(\|u_0\|_{H_0^1(\Omega)}^2+\|f\|^2_{L_2(0,T^*;{L_2(\Omega)})}\right),
      \\
      \|u\|^2_{L^\infty(0,T^*;H_0^1(\Omega))}&\leq C_3\left(\|u_0\|_{H_0^1(\Omega)}^2+\|f\|^2_{L_2(0,T^*;{L_2(\Omega)})}\right),
  \end{align*}
  where $C_1$, $C_2$, and $C_3$ are the constants from~\cite[Lemma~4.4]{Bachmayr21}. 
\end{theorem}

\begin{theorem}[Convergence of spatial discretizations]\label{thm:modelconvergence_space-discr}
  Let $\mathcal V_h$ be of the form~\eqref{eq:discretemodelspaces}.
  Let $u_{0,h}\in \mathcal M \cap \mathcal V_h$ define a sequence that converges to $u_0$ in ${H_0^1(\Omega)}$ as $h \searrow 0$ and let $u_0$ have positive ${L_2(\Omega)}$-distance $\sigma$ to the relative boundary $\overline{\mathcal M}^\mathsf w\setminus \mathcal M$.
  Then there exists a constant $c > 0$ independent of $\sigma$ and a constant $h_0>0$ such that there is a unique  
  $u_h$ in $W(0,T^*;{H_0^1(\Omega)},{L_2(\Omega)})\cap L_\eta(0,T^*;{H_0^1(\Omega)})$
  that solves \Cref{problem discrete} on the time interval $[0,T^*]$ when $T^* < \sigma^2/c$ for all $h\leq h_0$. Furthermore, $u_h$ converges to the unique solution $u$ of \Cref{modelproblem} in $W(0,T^*;{H_0^1(\Omega)},{L_2(\Omega)})\cap L_\eta(0,T^*;{H_0^1(\Omega)})$ weakly in $L_2(0,T^*;{H_0^1(\Omega)})$ and strongly in $C(0,T^*;{L_2(\Omega)})$, while the weak derivatives $u_h'$ converge weakly to $u'$ in $L_2(0,T^*,{L_2(\Omega)})$.
  \end{theorem}

\begin{theorem}[Stability]\label{thm:stabilitymodel}
  Let $u,v \in W(0,T^*;{H_0^1(\Omega)},{L_2(\Omega)})$ be two solutions of \Cref{modelproblem} on a time interval $[0,T^*]$ corresponding to right-hand sides $f,g \in L_2(0,T;{L_2(\Omega)})$ and initial values $u_0, v_0 \in \mathcal M$, respectively. Assume that the continuous representatives $u,v\in C(0,T^*;{L_2(\Omega)})$ have pointwise positive ${L_2(\Omega)}$-distance to $\overline{\mathcal M}^w \setminus \mathcal M$  of at least~$\sigma$. 
  Then for any $\varepsilon >0$,
  \[
  \|u(t)-v(t)\|_{L_2(\Omega)}^2\leq \left(\|u_0-v_0\|_{L_2(\Omega)}^2 +\frac{1}{\varepsilon}\int_0^t\|f(s)-g(s)\|_{L_2(\Omega)}^2\, ds\right)\exp(\Lambda(t)+\varepsilon t),
  \]
  where 
  \begin{multline*}
    \Lambda(t)\coloneqq 2\kappa\int_0^t \|u'(s)\|_{L_2(\Omega)}+\|v'(s)\|_{L_2(\Omega)}
    +
    \gamma\left(\|u(s)\|_{H_0^1(\Omega)}^\eta+\|v(s)\|_{H_0^1(\Omega)}^\eta\right)
    \\
    +
    \|f(s)\|_{L_2(\Omega)}+\|g(s)\|_{L_2(\Omega)}\, ds
     < \infty    
  \end{multline*}
  with $\kappa=\kappa(\sigma)=\frac{\sqrt{d+c^2}}{\sigma}$ from \Cref{prop: curvatur estimate 2}.
\end{theorem}

\begin{appendix}

\section{Proofs of \Cref{thm: manifold} and \Cref{prop: TT curvature}}

\begin{proof}[Proof of \Cref{thm: manifold}]
  Ad~(i).~We fix a particular $X_* = C_* \times_1 U_*^1 \times_2 \dots \times_d U_*^d$. The construction of the submersion follows~\cite{Steinlechner16}, where this has been done for manifolds of fixed multilinear rank in finite-dimensional tensor spaces. Here we additionally have to take the constraint $C \in \mathcal{M}_\mathrm{c}$ for the coefficient tensor into account. In the following, $\mathcal O$ is an open neighborhood of $X_*$ in $\ell_2(\mathbb N^d)$ that can always be chosen sufficiently small to ensure that all maps are well defined.
  
  Since $X_* \in \mathcal M$, the matricizations $M^{\mu}_{X_*}$ admit low-rank decompositions~\eqref{eq: representation T_u}, which can be written in matrix product form as 
  \[
   M^\mu_{X_*} = U_*^\mu (V_*^\mu)^\T.
  \]
  Here the columns of $U_*^\mu$ and $V_*^\mu$ are bases of the minimal subspaces $\mathcal U^\mu_* \subset \ell_2(\mathbb N)$ and $\mathcal V_*^\mu \subset \ell_2(\mathbb N^{d-1})$, respectively. For $X \in \mathcal M$ sufficiently close to $X_*$, it will be useful to define a particular basis $U_X^\mu$ for the $\mu$-th minimal subspace $\mathcal U^\mu$, $\mu=1,\ldots,d$, as a continuous function of $X$. To this end, we choose
  \begin{equation}\label{eq: basis U_X}
   U^\mu_X =  M_X^{\mu} (V_*^\mu)_+^\T,
  \end{equation}
  where $Y_+ = [Y^\T Y]^{-1} Y^\T$ denotes the pseudoinverse of a matrix with full column rank. Then $U^\mu_X$ has full column rank for $X$ close enough to $X_*$, which follows from $U^\mu_X \to U^\mu_*$ for all $\mu$ for $X \to X_*$ and the lower semicontinuity of the rank. As a result, every $X$ in the neighborhood of $X_*$ can be written as
  \[
  X = C_X \times_1 U_X^1 \times_2 \dots \times_d U_X^d.
  \]
  Moreover, we can assume that for $\mu=1,\dots,d$, the $r_\mu \times r_\mu$ matrices $(U_*^\mu)_+ U_X^\mu$ are invertible (again, since $U^\mu_X \to U^\mu_*$ for all $\mu$ for $X \to X_*$). We then also consider
  \begin{equation}\label{eq: barC_X}
   \bar C_X = X \times_1 (U_*^1)_+ \times_2 \dots \times_d (U_*^d)_+ = C_X \times_1  (U_*^1)_+ U_X^1 \times_2 \dots \times_d (U_*^d)_+ U_X^d .
  \end{equation}
  Clearly, $\bar C_{X_*} = C_*$. Noting that by~\eqref{eq: invariance Mhat} the condition $C_X \in \mathcal{M}_\mathrm{c}$ in~\eqref{eq: Tucker format} is independent under invertible changes of basis, we arrive at the following local description of $\mathcal M$:
  \begin{equation}\label{eq: local description}
   \mathcal M \cap \mathcal O = \{ X \in \mathcal O \colon \bar C_X \in \mathcal{M}_\mathrm{c}, \ \rank(M_X^\mu) = r_{\mu} \text{ for $\mu = 1,\dots,d$} \}.
  \end{equation}
  
  We next describe the constraints as preimages of smooth maps. We begin with the constraint $\bar C_X \in \mathcal{M}_\mathrm{c}$. Since $\mathcal{M}_\mathrm{c}$ is assumed to be an embedded submanifold of $\R^{r_1 \times \dots \times r_d}_*$, there exists a submersion $\phi$ from an open neighborhood of $C_* \in \mathcal{M}_\mathrm{c}$ to $\R^q$ (here $q$ is the co-dimension of $\mathcal{M}_\mathrm{c}$) such that the conditions $C \in \mathcal{M}_\mathrm{c}$ and $\phi(C) = 0$ are equivalent in this neighborhood. Considering
  \[
   g_0  : \mathcal O \to \R^q, \quad X \mapsto \phi(\bar C_X),
  \]
  we then locally have $\bar C_X \in \mathcal{M}_\mathrm{c}$ if and only if $g_0(X) = 0$. 
  
  We now consider the rank constraints in~\eqref{eq: local description}, which first will be reformulated. Let $P_{\mathcal U_*^\mu} = U_*^\mu (U_*^\mu)_+$ denote the orthogonal projections on the $\mu$-th minimal subspaces of $X_*$, and let
  \[
  \mathbb P_\mu = P_{\mathcal U_*^1} \otimes \dots \otimes  P_{\mathcal U_*^{\mu}} \otimes \mathrm{id}_{\mu+1} \otimes \dots \otimes \mathrm{id}_d,
  \]
  with the convention $\mathbb P_0 = \mathrm{id}$. In the following we consider tensors $\mathbb P_{\mu-1}(X)$, that is, orthogonal projections of $X$ onto the subspaces $\mathcal U_1^* \otimes \dots \otimes \mathcal U_*^{\mu-1} \otimes \ell_2(\mathbb N^{d-\mu+1})$. In particular, we claim that for any fixed $1 \le \nu \le d$ there exists a neighborhood of $X_*$ in which the  condition $\rank(M_{X}^\mu) = r_\mu$ for $\mu=1,\dots,\nu$ is equivalent with $\rank(M_{\mathbb P_{\mu-1}(X)}^\mu) = r_\mu$ for all $\mu=1,\dots,\nu$. This is shown by induction over $\nu$. For $\nu=1$ the statement is trivial since $\mathbb P_0 = \mathrm{id}$. In the induction step $\nu-1 \to \nu$, it suffices to show that in some neighborhood of $X_*$, any $X$ satisfying $\rank(M_X^\mu) = r_\mu$ for $\mu = 1,\dots,\nu-1$ also satisfies $\rank(M_X^\nu) = \rank(M_{\mathbb P_{\nu-1}(X)}^\nu)$. Any such $X$ lies in a subspace $\mathcal U^1 \otimes \dots \otimes \mathcal U^{\nu-1} \otimes \ell_2(\mathbb N^{d - \nu +1})$, where $\mathcal U^\mu$ are the minimal $r_\mu$-dimensional subspaces of $X$. We choose a neighborhood of $X_*$ in which the restrictions of all $P_{\mathcal U_*^\mu}$ to $\mathcal U^\mu$ are necessarily invertible maps between $\mathcal U^\mu$ and $\mathcal U_*^\mu$ (which is equivalent with $(U^\mu_*)_+  U^\mu_X$ being invertible). 
  Hence in this neighborhood the projection $\mathbb P_{\nu-1}$ is a tensor product of invertible operators between $\mathcal U^1 \otimes \dots \otimes \mathcal U^{\nu-1} \otimes \ell_2(\mathbb N^{d - \nu +1})$ and $\mathcal U^1_* \otimes \dots \otimes \mathcal U^{\nu-1}_* \otimes \ell_2(\mathbb N^{d - \nu +1})$, which hence leaves all matricization ranks invariant. Hence, for such $X$, we obtain that $\rank(M_{X}^\nu) = \rank(M_{\mathbb P_{\nu-1} X}^\nu)$, which completes the induction.
  
   Applying the above equivalence with $\nu = d$ allows us to replace the conditions $\rank(M_X^\mu) = r_{\mu}$ in~\eqref{eq: local description} with $\rank(M_{\mathbb P_{\mu-1}(X)}^\mu) = r_\mu$ for $\mu =1 ,\ldots, d$. These latter conditions are now handled via Schur complements as follows. Note that
  \[
  M_{\mathbb P_{\mu-1}(X)}^\mu \in \ell_2(\mathbb N) \otimes [\mathcal U_*^1 \otimes \dots \otimes \mathcal U_*^{\mu-1} \otimes \ell_2(\mathbb N^{d -\mu})].
  \]
  We consider orthogonal decompositions
  \[
  \ell_2(\mathbb N) = \mathcal U_*^\mu \oplus  (\mathcal U_*^\mu)^\perp
  \]
  and
  \[
  \mathcal U_*^1 \otimes \dots \otimes \mathcal U_*^{\mu-1} \otimes \ell_2(\mathbb N^{d -\mu}) = \mathcal V_*^\mu \oplus \mathcal W^\mu_*,
  \]
  which is possible since $\mathcal V_*^\mu$ is even contained in the smaller subspace of $\mathcal U_*^1 \otimes \dots \otimes \mathcal U_*^{\mu-1} \otimes \mathcal U_*^{\mu+1} \otimes \dots \otimes \mathcal U_*^d$ (which in turn follows from $X_* \in \mathcal U_*^1 \otimes \dots \otimes \mathcal U_*^d$). Hence in this notation
  \[
  M_{\mathbb P_{\mu-1}(X)} \in [\mathcal U_*^\mu \oplus (\mathcal U_*)^\perp] \otimes [\mathcal V_*^\mu \oplus \mathcal W^\mu_*].
  \]
  By applying a block decomposition of $M_{\mathbb P_{\mu-1}(X)}^\mu$ into the four corresponding parts,
  \begin{align*}
   Q^\mu_X &= P_{\mathcal U^\mu_*} M^\mu_{\mathbb P_{\mu-1}(X)} P_{\mathcal V^\mu_*} \in \mathcal U^\mu_* \otimes \mathcal V^\mu_*, \\ 
  R^\mu_X &= P_{\mathcal U^\mu_*} M^\mu_{\mathbb P_{\mu-1}(X)} P_{\mathcal W^\mu_*} \in \mathcal U^\mu_* \otimes \mathcal W^\mu_*, \\
  S^\mu_X &= (\id - P_{\mathcal U^\mu_*})M^\mu_{\mathbb P_{\mu-1}(X)} P_{\mathcal V^\mu_*} \in (\mathcal U^\mu_*)^\perp \otimes \mathcal V^\mu_*, \\
  T^\mu_X &= (\id - P_{\mathcal U^\mu_*} ) M^\mu_{\mathbb P_{\mu-1}(X)}  P_{\mathcal W^\mu_*} \in (\mathcal U^\mu_*)^\perp \otimes \mathcal W^\mu_*,
  \end{align*}
  we can consider the Schur complement functions 
  \begin{equation}\label{eq: definition g_mu}
  g_\mu : \mathcal O \to  (\mathcal U_*^\mu)^\perp \otimes \mathcal W_*^\mu, \quad X \mapsto T_X^\mu - S_X^\mu (Q_X^\mu)^{-1} R_X^\mu.
  \end{equation}
  Note that $Q_X^\mu$ is indeed invertible (as an $r_\mu \times r_\mu$ matrix in $\mathcal U^\mu_* \otimes \mathcal V^\mu_*$) for $X$ close enough to $X_*$, since $Q_{X_*}^\mu = P_{\mathcal U^\mu_*} M^\mu_{X_*} P_{\mathcal V^\mu_*}$ is invertible. As for finite matrices, we then have $g_\mu(X) = 0$ if and only if $\rank(M_{\mathbb P_{\mu-1}(X)}^\mu) = r_\mu$.
  
  Defining
  \[
  g = (g_0,g_1,\dots,g_d) \colon \mathcal O \to \R^q \times [(\mathcal U_*^1)^\perp \otimes \mathcal W_*^1] \times \dots \times [(\mathcal U_*^d)^\perp \otimes \mathcal W_*^d]
  \]
  we conclude from all the previous considerations that~\eqref{eq: local description} can be written as
  \[
  \mathcal M \cap \mathcal O = g^{-1}(0).
  \]
  
  We need to show that $g$ is a submersion in $X_*$, that is, $g'(X_*)$ is surjective. First note that for $\mu = 1,\dots,d$ we have
  \begin{equation}\label{eq: derivative g_mu}
  g_\mu'(X_*)[H] = T_{H}^\mu = (\id - P_{\mathcal U^\mu_*} ) M^\mu_{\mathbb P_{\mu-1}(H)} P_{\mathcal W^\mu_*},
  \end{equation}
  that is, $g_\mu'(X_*)$ is the orthogonal projection (of the $\mu$-th matricization) onto the subspace $(\mathcal U_*^\mu)^\perp \otimes \mathcal W_*^\mu$. This follows by applying a product rule to~\eqref{eq: definition g_mu} and noting that $R^\mu_{X_*} = 0$ and $S^\mu_{X_*} = 0$. When viewed as subspaces of $\ell_2(\mathbb N^d)$, the subspaces $(\mathcal U_*^\mu)^\perp \otimes \mathcal W_*^\mu$ are mutually orthogonal to each other, since they are contained in the pairwise orthogonal subspaces $\mathcal U_*^1 \otimes \dots \otimes \mathcal U_*^{\mu-1} \otimes (\mathcal U_*^\mu)^\perp \otimes \ell_2(\mathbb N^{d-\mu})$, respectively. Moreover, all of them are orthogonal to the subspace $\mathcal U_*^1 \otimes \dots \otimes \mathcal U_*^d$. Regarding $g_0$, note that
  \begin{equation}\label{eq: g0 on subspace}
  g_0(C \times_1 U^1_* \times_2 \dots \times_d U^d_*) = \phi(C)
  \end{equation}
  (again with $U^\mu_{X_*} = U^\mu_*$), which shows that already the restriction of $g_0$ to $\mathcal U_*^1 \otimes \dots \otimes \mathcal U_*^d$ is a submersion in $X_*$, since $\phi'(C_*)$ is surjective to $\R^q$. It is now easy to conclude from these facts that $g'(X_*)$ is altogether surjective.
  
  By the local submersion theorem in Hilbert space, see~\cite[Thm.~73.C]{Zeidler88}, $\Mc \cap \mathcal O = g^{-1}(0)$ is a smooth submanifold of $\mathcal H$. The tangent space $T_{X_*} \mathcal M$ at $X_*$ is the null space of the $g'(X_*)$. The proof of part (i) is therefore completed.
  
  \bigskip
  
  Ad~(ii).~The existence of the continuously Fr\'echet-differentiable homeomorphism $\varphi$ of the asserted form is a consequence of Ljusternik's submersion theorem as stated in~\cite[Thm.~43.C]{Zeidler85}. To show that (in a possibly smaller neighborhood around zero) $\varphi$ is also an immersion, that is, $\varphi'(\xi) \colon T_{X_*} \mathcal M \to T_{\varphi(\xi)} \mathcal M$ is injective and its range splits, it suffices to show that there exists $c>0$ such that $\| \varphi'(\xi)h\| \ge c \| h \|$ for all $h \in T_{X_*} \mathcal M$. This, however, follows immediately from the continuity of $\varphi'$ and $\varphi'(0)h = h$. By definition, $\varphi$ is therefore a local embedding~\cite[Def.~73.43]{Zeidler88}.
  
  \bigskip
  
  Ad~(iii). ~Let $\xi$ be an element of the form~\eqref{eq: tangent vectors} (but at $X_* = C_* \times_1 U^1_* \times_2 \dots \times_d U^d_*$). Since $\dot C \in T_{C_*} \mathcal{M}_\mathrm{c}$, there exists a curve $C(t)$ in $\mathcal{M}_\mathrm{c}$ such that $C(0) = C_*$ and $C'(0) = \dot C$. For small enough $t$, 
  \[
  X(t) = C(t) \times_1 (U_*^1 + t \dot U^1) \times_2 \dots \times_d (U_*^d + t \dot U^d)
  \]
  then defines a curve in $\mathcal M$ because $U_*^\mu + t \dot U^\mu$ has full column rank for $\mu = 1,\ldots,d$. Obviously $X(0) = X_*$, and by multilinearity it is easily seen that $X'(0) = \xi$, which shows $\xi \in T_{X_*} \mathcal M$.
  
  In order to show that all tangent vectors are of the form~\eqref{eq: tangent vectors}, let $\xi \in T_{X_*} \mathcal M$ and a corresponding curve $X(t) \in \mathcal M$ with $X(0) = X_*$ and $X'(0) = \xi$ be given. For small enough~$t$ we represent $X(t)$ in the particular bases $U^\mu_{X(t)}$ defined in~\eqref{eq: basis U_X} as
  \[
  X(t) = C_{X(t)} \times_1 U^1_{X(t)} \times_2 \dots \times_d U^d_{X(t)},
  \]
  where $C_{X(t)}$ is in $\mathcal{M}_\mathrm{c}$. Clearly, the curves $t \mapsto U^\mu_{X(t)}$ are smooth. It implies that for small enough $t$ the pseudoinverses $t \mapsto (U^\mu_{X(t)})_+$ are also smooth functions. It then follows from~\eqref{eq: barC_X}, by applying an inverse transformation, that also $t \mapsto C_{X(t)}$ is a smooth curve, since $\bar C_{X(t)}$ is. By the product rule we then get that
  \begin{equation}\label{eq: expression xi}
  \xi = X'(0) = \tilde C \times_1 U^1_* \times_2 \dots \times_d U^d_* + C_* \times_1 \tilde{U}^1 \times_2 \dots \times_d U^d_* + \dots + C_* \times_1 U^1_* \times_2 \dots \times_d \tilde{U}^d,
  \end{equation}
  where $\tilde C \in T_{C_*} \mathcal{M}_\mathrm{c}$ is the derivative of $t \mapsto C_{X(t)}$ in $t=0$, and $\tilde U^\mu \in (\ell_2(\mathbb N))^{r_\mu}$ is the derivative of $t \mapsto U^\mu_{X(t)}$ in $t =0$ for $\mu=1,\ldots,d$. By decomposing every column of $\tilde U^\mu$ into the span of $U^\mu_*$ and its orthogonal complement, we can write
  \[
  \tilde U^\mu = U^\mu_* S_\mu + \dot U^\mu_*,
  \]
  where $S_\mu$ is some $r_\mu \times r_\mu$ matrix and $(U^\mu_*)^\T \dot U^\mu_* = 0$. Expanding the expression~\eqref{eq: expression xi} we then have
  \[
  \xi = K \times_1 U^1_* \times_2 \dots \times_d U_*^d + C_* \times_1 \dot{U}^1_* \times_2 \dots \times_d U^d_* + \dots + C_* \times_1 U^1_* \times_2 \dots \times_d \dot{U}^d_*
  \]
  where
  \[
  K = \tilde C + C_* \times_1 S_1 \times_2 \id \times_3 \dots \times_d \id + \dots + C_* \times_1 \id \times_2 \dots \times_d S_d.
  \]
  It remains to show that $K \in T_{C_*} \mathcal{M}_\mathrm{c}$ to conclude that $\xi$ is of the asserted form~\eqref{eq: tangent vectors}. Since $T_{C_*} \mathcal{M}_\mathrm{c}$ is a linear space it suffices to show this for every term. For $\tilde C$ there is nothing to show. The sum of the remaining terms equals the derivative of the curve
  \[
  C(t) = C_* \times_1 (\id + t S_1) \times_2 \dots \times_d (\id + t S_d)
  \]
  at $t=0$, which for small enough $t$ lies in $\mathcal{M}_\mathrm{c}$ by the invariance condition~\eqref{eq: invariance Mhat}. Hence $C'(0) \in T_{C_*} \mathcal{M}_\mathrm{c}$.
  \end{proof}

  \begin{proof}[Proof of \Cref{prop: TT curvature}]
    Let $X^{\{1,\ldots, \mu\}} \in \R^{N_1\cdots N_\mu \times B_{\mu+1} \cdots N_d} $ be a matricization of $X$. We can decompose 
    \[
     X^{\{1,\ldots, \mu\}} 
     =
     (U_1\otimes \id_{N_2\cdots N_\mu})\cdots(U_{\mu-1}\otimes \id_{N_\mu}) U_\mu 
     \Sigma_\mu 
     V_{\mu+1}^\T
    \cdots (\id_{N_{\mu+1}\cdots N_{d-1}} \otimes V_d^\T),
    \]
    with $U_\nu^\T U_\nu^{} = \id_{r_{\nu}}$ and $V_\nu^\T V_\nu^{} = \id_{r_{\nu}}$. Furthermore, we define the spaces $\mathcal U_{\{1,\ldots, \mu\}}$ and  $\mathcal V_{\{\mu+1,\ldots, d\}}$ via their respective orthonormal bases
    \[
     U_{\{1,\ldots, \mu\}} =  (U_1\otimes \id_{N_2\cdots N_\mu})\cdots(U_{\mu-1}\otimes \id_{N_\mu}) U_\mu 
     \] 
     and 
     \[
     V_{\mu+1,\ldots, d}^\T = V_{\mu+1}^\T
    \cdots (\id_{N_{\mu+1}\cdots N_{d-1}} \otimes V_d^\T).
    \]
    The projection $P_X$ can be decomposed as
   \[
     P_X = P^X_1 +\ldots +P^X_d,
   \]
     where 
   \[
     P^X_\mu = (P_{\mathcal{U}_{\{1,\ldots, \mu-1\}}}\otimes \id_{N_\mu}-    P_{\mathcal{U}_{\{1,\ldots, \mu\}}}) \otimes P_{\mathcal V_{\{\mu+1,\ldots, d\}}}
   \]
   for $\mu = 1,\ldots, d-1$ and 
   \[
     P^X_d = P_{\mathcal{U}_{\{1,\ldots, d-1\}}} \otimes \id_{N_d};
   \]
   see, e.g. \cite{Lubich15} or \cite[Section 9.3.4]{UschmajewVandereycken:20}.
   Let $\tilde {\mathcal U}_{\{1,\ldots, \mu\}}$ and  $\tilde{\mathcal  V}_{\{\mu+1,\ldots, d\}}$ be the analogous spaces for~$Y$.
   Then by \eqref{eq: projector difference},
   \begin{align*}
     \|P^X_\mu- P^Y_\mu\|
     &=
     \|
     (P_{\mathcal{U}_{\{1,\ldots, \mu-1\}}}\otimes \id_{N_\mu}-    P_{\mathcal{U}_{\{1,\ldots, \mu\}}}) \otimes P_{\mathcal V_{\{\mu+1,\ldots, d\}}}\\
     &\quad -
     (P_{\tilde{\mathcal{U}}_{\{1,\ldots, \mu-1\}}}\otimes \id_{N_\mu}-    P_{\tilde{\mathcal{U}}_{\{1,\ldots, \mu\}}}) \otimes P_{\tilde{\mathcal{V}}_{\{\mu+1,\ldots, d\}}}
     \|\\
     &\leq
     \|
     (P_{\mathcal{U}_{\{1,\ldots, \mu-1\}}}-    P_{\tilde{\mathcal{U}}_{\{1,\ldots, \mu-1\}}}) \otimes \id_{N_\mu}\otimes P_{\mathcal V_{\{\mu+1,\ldots, d\}}}\|\\
     &\quad +
     \|
     P_{\tilde{\mathcal{U}}_{\{1,\ldots, \mu-1\}}} \otimes \id_{N_\mu}\otimes (P_{\mathcal V_{\{\mu+1,\ldots, d\}}}-P_{\tilde{\mathcal V}_{\{\mu+1,\ldots, d\}}})\|\\
     &\quad +
     \|
     (P_{\mathcal{U}_{\{1,\ldots, \mu\}}}-    P_{\tilde{\mathcal{U}}_{\{1,\ldots, \mu\}}}) \otimes P_{\mathcal V_{\{\mu+1,\ldots, d\}}}\|\\
     &\quad +
     \|
     P_{\tilde{\mathcal{U}}_{\{1,\ldots, \mu\}}} \otimes  (P_{\mathcal V_{\{\mu+1,\ldots, d\}}}-P_{\tilde{\mathcal V}_{\{\mu+1,\ldots, d\}}})\|\\
     &\leq
     \frac{4}{\sigma}\|X-Y\|
   \end{align*}
 holds and the first desired inequality readily follows. For the other inequalities, we use the decomposition of the identity matrix
 \begin{multline*}
 \id = P_{\mathcal{U}_{\{1,\ldots, d-1\}}} \otimes \id_{N_d}
 + (P_{\mathcal{U}_{\{1,\ldots, d-2\}}}\otimes \id_{N_{d-1}}- P_{\mathcal{U}_{\{1,\ldots, d-1\}}})\otimes \id_{N_d}  \\
 +\ldots+ (\id_{N_1}-P_{\mathcal{U}_{\{1\}}})\otimes \id_{N_2\dots N_d}
 \end{multline*}
 into orthogonal projections onto mutually orthogonal spaces. Then 
 \begin{align*}
   (\id -P_X)(X-Y)
   &=
   (P_{\mathcal{U}_{\{1,\ldots, d-2\}}}\otimes \id_{N_{d-1}}-    P_{\mathcal{U}_{\{1,\ldots, d-1\}}}) \otimes (\id_{N_{d-1}N_d} - P_{\mathcal V_{\{d\}}})(X-Y)\\
   &\quad +\ldots\\
   &\quad + 
   ({\id_{N_{1}}}-    P_{\mathcal{U}_{\{1\}}}) \otimes (\id_{N_{2}\dots N_d} - P_{\mathcal V_{\{2,\ldots,d\}}})(X-Y)\\
   &=
   (P_{\mathcal{U}_{\{1,\ldots, d-2\}}}\otimes \id_{N_{d-1}}-    P_{\mathcal{U}_{\{1,\ldots, d-1\}}}) \otimes ( P_{\tilde{\mathcal V}_{\{d\}}} - P_{\vphantom{\tilde{\mathcal V}_{\{d\}}}{\mathcal V_{\{d\}}}})(X-Y)\\
   &\quad +\ldots\\
   &\quad + 
   ({\id_{N_{1}}}-    P_{\mathcal{U}_{\{1\}}}) \otimes (P_{\tilde{\mathcal V}_{\{2,\ldots,d\}}} - P_{\vphantom{\tilde{\mathcal V}_{\{2,\ldots,d\}}} {\mathcal V_{\{2,\ldots,d\}}}})(X-Y)
 \end{align*}
 holds.
 Note that the operators map onto orthogonal subspaces. Hence, we get the desired estimate
 \[
   \|(\id -P_X)(X-Y)\|\leq \frac{\sqrt{d-1}}{\sigma}\|X-Y\|^2.
 \]
 By continuity of the projection and taking limits, the estimate also holds for $Y\in\overline{\mathcal{M}_\mathrm{c}}$.
 In a similar way, the second inequality follows. 
 \end{proof}
 
\end{appendix}

\subsection*{Acknowledgements}

M.B.\ acknowledges funding by Deutsche Forschungsgemeinschaft (DFG, German Research Foundation) -- Projektnummern 442047500, 501389786.
The work of H.E.~was funded by Deutsche Forschungs\-gemeinschaft (DFG, German Research Foundation) – Projektnummer 501389786.
The work of A.U.~was supported by Deutsche Forschungs\-gemeinschaft (DFG, German Research Foundation) – Projektnummer 506561557.

\bibliographystyle{plain}
\bibliography{BEUparabolic2}

\end{document}